\documentclass[a4paper,11pt]{article}
\usepackage{a4wide}
\usepackage{cite}
\usepackage[utf8]{inputenc}
\usepackage[T1]{fontenc}

\usepackage[colorlinks=true, pdfstartview=FitV, linkcolor=blue,citecolor=blue, urlcolor=blue]{hyperref}

\usepackage[french,english]{babel}
\usepackage{amsmath}
\usepackage{amscd}
\usepackage{amssymb}
\usepackage{amsrefs}
\usepackage{amsthm}
\usepackage{cases}
\usepackage{framed}
\usepackage{graphicx}
\usepackage{latexsym}
\usepackage{color}
\usepackage{array,multirow,makecell}

\setcellgapes{1pt}
\makegapedcells

\theoremstyle{plain} 
\newtheorem{thm}{Theorem}[section]
\newtheorem{prop}[thm]{Proposition}
\newtheorem{lemma}[thm]{Lemma} 
\newtheorem{assumption}[thm]{Assumption}

\theoremstyle{definition} 
\newtheorem{df}[thm]{Definition}

\theoremstyle{remark} 
\newtheorem{rmk}[thm]{Remark}

\renewcommand{\d}{\mathrm{d}}
\newcommand{\Div}{\mathrm{div}}

\newcommand{\R}{\mathbb{R}}

\newcommand{\ep}{\varepsilon}

\newcommand{\tp}{\tilde{p}}

\newcommand{\teta}{\tilde \eta}
\newcommand{\tq}{\tilde q}

\newcommand{\tgamma}{\widetilde \gamma}

\newcommand{\brho}{\bar \rho}
\newcommand{\bm}{\bar m}
\newcommand{\bu}{\bar{u}}

\newcommand{\ER}{\mathrm{Error}}
\title{Soft congestion approximation\\ to the one-dimensional constrained Euler equations}

\author{Roberta Bianchini\footnote{Sorbonne-Universit\'e, CNRS, Universit\'e de Paris, Laboratoire Jacques-Louis Lions (LJLL), F-75005 Paris, France \& Consiglio Nazionale delle Ricerche, IAC, via dei Taurini 19, I-00185 Rome, Italy; r.bianchini@iac.cnr.it} \ and Charlotte Perrin\footnote{Aix Marseille Univ, CNRS, Centrale Marseille, I2M, Marseille, France; charlotte.perrin@univ-amu.fr}}

\begin{document}
	
	\maketitle
	
	\begin{small}
		\begin{center}
			{\bf Abstract}
		\end{center}
		
		This article is concerned with the analysis of the one-dimensional compressible Euler equations with a singular pressure law, the so-called \emph{hard sphere equation of state}. The  result is twofold.
		First, we establish the existence of bounded weak solutions
		by means of a viscous regularization and refined compensated compactness arguments.
		Second, we investigate the smooth setting by providing a detailed description of the impact of the singular pressure on the breakdown of the solutions.
		In this smooth framework, we rigorously justify the singular limit towards the free-congested Euler equations, where the compressible (free) dynamics is coupled with the incompressible one in the constrained (i.e. congested) domain.

		\bigskip
		\noindent{\bf Keywords:} Compressible Euler equations, maximal packing constraint, singularity formation, singular limit, free boundary problem, compensated compactness.
		
		\medskip
		\noindent{\bf MSC:} 35Q35, 35L87, 35L81.
	\end{small}


\section{Introduction and main results}
The topic of this work is the analysis of the following one-dimensional compressible Euler equations
\begin{subnumcases}{\label{eq:sing-Euler}}
 \partial_t \rho +\partial_x m =0, \label{eq:mass-sing-Euler} \\
\partial_t m + \partial_x \left( \dfrac{m^2}{\rho}\right) + \partial_x p_\ep(\rho)=0, \label{eq:moment-sing-Euler}
\end{subnumcases}
where $\rho$ stands for the density and $m = \rho u$ for the momentum of the fluid, 
with $u$ the velocity of the fluid.
The originality of the model that we shall consider in this paper lies in the choice the pressure law $p_\ep$, which is supposed to satisfy the so-called \emph{hard-sphere equation of state} introduced by Carnahan and Starling in~\cite{carnahan1969}.  The latter is identified by means of the following conditions at $\ep>0$ fixed:
\begin{equation}\label{eq:description-pressure}
p_\ep \in \mathcal{C}^1([0,1)), \quad p_\ep(0) = 0, \quad p'_\ep(\rho)> 0 ~~\text{on}~ (0,1), \quad \lim_{\rho \to 1^-}p_\ep(\rho) = + \infty,
\end{equation}
where the physical meaning of the parameter $\ep>0$ is discussed below.
The class of equations in \eqref{eq:sing-Euler}-\eqref{eq:description-pressure} gained the interest of the mathematical community for the modeling of collective motions (see for instance \cite{maury2011} and \cite{degond2011}) and of dispersed mixtures like bubbly fluids or granular suspensions (see for instance \cite{harris1994}, \cite{feireisl2018}, \cite{ozenda2019}).
In the collective motion models, $\rho$ in \eqref{eq:sing-Euler} is the density of the crowd, while the pressure $p_\ep(\rho)$ is the cumulative response of short-range repulsive social forces preventing contacts between individuals. 
From the macroscopic viewpoint, the singularity of the pressure plays the role of a barrier, formally preventing the creation of congested regions where $\rho=1$.
The parameter $\ep$ models the strength of the repulsive forces. 

In the rest of the paper, the expression of the pressure term is explicitly chosen as follows
\begin{equation}\label{eq:sing-press-Euler}
p_\ep(\rho) = \ep \left(\dfrac{\rho}{1-\rho} \right)^\gamma + \kappa \rho^{\tgamma}
= p_{1,\ep}(\rho) + p_2(\rho),
\end{equation}
where $\ep>0$ is small and fixed, while the specific ranges of $\gamma, \tgamma>1$ and $\kappa\ge 0$ will be discussed later on. 
The pressure is thus split into two parts: the first one $p_{1,\ep}$ dictates the singular behavior close to the maximal density constraint, while $p_2$ is the classical equation of state for isentropic gases and models additional non-singular effects.
For instance, shallow water flows can be described by system \eqref{eq:sing-Euler} (the so-called \emph{shallow water} or \emph{Saint-Venant equations}), where the variable $\rho$ is replaced by the height of the flow $h$, and $p_2$ is the hydrostatic part of the pressure due to gravity, namely $p_2(h)= gh^2/2$.\\
A heuristic reasoning shows that the solutions $(\rho_\ep,m_\ep)$ to system \eqref{eq:sing-Euler} coupled with the equation of state \eqref{eq:sing-press-Euler} converge as $\ep\rightarrow 0$ towards the solutions $(\rho,m)$ to the \emph{free-congested Euler equations}
\begin{subnumcases}{\label{eq:lim-Euler}}
\partial_t \rho + \partial_x m = 0, \\
\partial_t m + \partial_x \left( \frac{m^2}{\rho} \right) + \partial_x p + \kappa \partial_x \rho^{\tgamma} = 0, \\
0 \leq \rho \leq 1, ~(1-\rho)p = 0, ~ p \geq 0, \label{eq:lim-Euler-excl}
\end{subnumcases}
where the pressure $p$ is the limit (in a sense that will be clarified later on) of $p_{1,\ep}(\rho_\ep)$.
The above system is a hybrid model describing both regions where the density is ``free'', in the sense that $\rho<1$ and $p=0$, and constrained regions where the density is saturated $\rho=1$ and $p$ activates itself.
From the mathematical viewpoint, the pressure $p$ can be seen as a Lagrange multiplier associated to the incompressibility constraint
\begin{equation}\label{eq:divu}
\partial_x u = 0 \quad \text{in} \quad \{\rho =1\}.
\end{equation}
Following the terminology introduced by Maury in \cite{maury2011}, compressible systems with singular constitutive laws like \eqref{eq:sing-Euler}-\eqref{eq:description-pressure} are called \emph{soft congestion models}, whereas free-congested systems of type \eqref{eq:lim-Euler} are called \emph{hard congestion models}.
It is worth pointing out that, unlike the standard formulation of free-boundary problems
, in \eqref{eq:lim-Euler} there is no explicit equation for the evolution of the interface between the free domain and the congested one, which is indeed implicitly encoded in the exclusion relation \eqref{eq:lim-Euler-excl}.\\
The limit system \eqref{eq:lim-Euler}, with $\kappa=0$, has been heuristically introduced  by Brenier et al. in \cite{brenier2000} as an asymptotic model for two-phase (gas-liquid or solid-liquid) flows when the ratio between the characteristic densities of the two phases is very small (or conversely very large).
The existence of global weak solutions to system \eqref{eq:lim-Euler} with $\kappa=0$ has been established by Berthelin in \cite{berthelin2002} (see also \cite{perrin2018} for a closely related model and \cite{berthelin2017} for an extension to the multi-D case) and numerical approaches based on optimal transport are proposed in~\cite{preux2016}. \\
To our knowledge, the rigorous proof of the convergence of solutions to \eqref{eq:sing-Euler}-\eqref{eq:sing-press-Euler} towards solutions of \eqref{eq:lim-Euler} is a largely open question, previous studies in the literature dealing only with the formal link between the two models.
For instance, Degond et al. in~\cite{degond2011} take advantage of this formal limit to provide a new numerical scheme for the free boundary problem \eqref{eq:lim-Euler}.
Interestingly, the analysis of the asymptotic behavior as $\ep \to 0$ of the solutions of the Riemann problem associated to \eqref{eq:sing-Euler} is also carried out in \cite{degond2011}. 
In~\cite{bresch2017} Bresch and Renardy analyse the shock formation at the interface between the congested region where $\rho=1$ and the free region where $\rho < 1$.
In that case, the heuristic connection between \eqref{eq:sing-Euler} and \eqref{eq:lim-Euler} plays again a crucial role in identifying numerically the formation of these shocks when the congestion constraint is reached.
As a matter of facts, the asymptotic limit for $\ep \to 0$ is better understood in the viscous case, that is the Navier-Stokes equations, where the viscosity term $ -\nu \partial_{xx} u$ is added to the momentum equation \eqref{eq:moment-sing-Euler}.
The interested reader is referred to \cite{peza2015}, where the behavior as $\ep \to 0$ for the multi-dimensional Navier-Stokes equations with a hard-sphere potential is investigated, and to the survey paper \cite{perrin2018-survey}, which provides a precise picture on the related state of the art. 
Finally, we remark that the asymptotics $\ep \to 0$ also shares some features with other kinds of singular limits for the compressible Euler equations, as the vanishing pressure limit~\cite{chen2003formation} and the low Mach number limit~\cite{colombo2016}.

\bigskip
The aim of this work is twofold. First, we construct global-in-time weak ($L^\infty$) solutions to the singular compressible system \eqref{eq:sing-Euler}, where the parameter $\ep> 0$ 
is fixed. Then we turn to the smooth framework.
In the latter, our first goal is to provide a precise description of how the solution breakdown is influenced by the hard-sphere potential.
Next, we aim at rigorously justifying the convergence $\ep \to 0$ of smooth solutions to \eqref{eq:sing-Euler} towards (weak) solutions to system \eqref{eq:lim-Euler}.

\bigskip
Before stating our main results, we provide a brief overview of the existing literature on the related mathematical setting.   

The investigation on the existence of weak solutions to the one-dimensional compressible Euler equations for arbitrarily large $L^\infty$ initial data started with the work of DiPerna~\cite{diperna1983}, where the validity of the vanishing viscosity approximation (an artificial viscosity is added to the system and then sent to zero) was established for the first time, by means of the \emph{compensated compactness method} introduced in the late 70's by Tartar and Murat (see [Chapter 5, \cite{evans1990}] and references therein). 
The general idea is the following. 
Weak-* convergence of the sequence of approximate solutions is ensured by the \emph{invariant region method} which provides 
$L^\infty$ bounds (uniform with respect to the viscosity parameter) on the sequence.
However, this weak convergence is not enough to be allowed to pass to limit in the nonlinear terms of the equations, namely the convective term $m^2/\rho$ and the pressure, because of potential high-frequency oscillations of the approximate solutions.
The main core of the compensated compactness method consists exactly in showing that the mechanism of entropy dissipation actually quenches the high-frequency oscillations, so enforcing strong convergence of the approximate solutions. 
DiPerna in \cite{diperna1983} started the investigation with the case of polytropic gases $p(\rho) = \kappa \rho^\gamma$, where $\gamma = 1 +2/n$, $n\geq 5$ odd. Later, the study was extended to all values $\gamma \in [1, +\infty)$, as a cumulative result due to many authors and several steps (see \cite{chen2000} for a review, and ~[Chapter 16, \cite{dafermos2000}]).  
In DiPerna's work, and in most of the following studies as well, the compensated compactness arguments rely on an explicit formula for the fundamental solution (the so-called \emph{entropy kernel}) of the entropy equation, which is available in the polytropic case i.e. for $p(\rho)= \kappa \rho^\gamma$. 
An alternative strategy for providing weak solutions for polytropic gases is due to Lu, see [Chapter 8, \cite{lu2002}]. Rather than  relying on the whole family of (infinite) entropies generated by the entropy kernel, this argument makes a smart use of only four entropy-entropy flux pairs, whose expressions are explicit.
This method is employed in \cite{lu2002} to deal with polytropic gases with $\gamma \in (1,3)$. 
The case of more general pressure laws has been tackled by Chen and LeFloch in \cite{chenlefloch2000} and \cite{chen2003}. 
The authors showed the existence of weak solutions to the one-dimensional compressible Euler equations for arbitrarily large $L^\infty$ initial data with non-singular pressure laws $p=p(\rho)\in \mathcal{C}^3(0,+\infty)$, which roughly speaking behave like $\rho^\gamma$, $\gamma \in (1,3)$ close to vacuum. 
Their result is based on a refined analysis of the singularities and takes advantage of some cancellation for the (non-explicit) entropy kernel.
In this paper, we present a new result of existence of weak solutions in the case of a singular pressure law $p_\ep$ depending on a small fixed parameter $\ep>0$.
Although a general result for hard-sphere potentials satisfying  \eqref{eq:description-pressure} could follow as an application of the framework due to Chen and LeFloch (see \cites{chenlefloch2000, chen2003}), we chose to provide a more explicit 
proof of the compactness of the solutions for the specific pressure law \eqref{eq:sing-press-Euler} (with $\kappa=0$ and $\gamma \in (1,3]$).
This proof makes use of four explicit entropy-entropy flux pairs in the spirit of Lu's work, [Chapter 8, \cite{lu2002}] and allows us to keep track of the singular parameter $\ep$ throughout all the computations. 

\medskip
The second part of this paper deals with the framework of smooth ($C^1$) solutions. 
The analysis of one-dimensional gas dynamics in this setting has a long history, which started with Lax~\cite{lax1964} in the 60's and was further developed by Chen and his coauthors in a series of recent papers (see for instance \cite{chen2017}, \cite{chen2013}, and references therein). 
In his original paper \cite{lax1964} on $2\times 2$ strictly hyperbolic systems, Lax considers initial data which are small perturbations of a constant state and
 shows that if these initial data contain some ``compression'' (
 in a sense precised below) then the corresponding smooth solutions develop singularities (i.e. blow-up of the gradient of the solution) in finite time; otherwise the solutions are global in time.
This result applies in particular to the compressible Euler equations, more precisely to its reformulation in Lagrangian coordinates, the so-called \emph{p-system} (see system \eqref{eq:p-system} below), in the context of small initial data.
The appearance of singularities for large initial data was instead an open question until the recent work of Chen et al.~\cite{chen2017}. 
They show that singularity formations occur in the p-system with the pressure law $p(\rho)=\kappa \rho^\gamma$ and $\gamma > 1$, if the initial datum (whose size is arbitrarily chosen) contains some compression (in the sense of Definition \ref{def:noncompressive-case}). 
Otherwise, if the initial datum is \emph{everywhere rarefactive} (see again Definition~\ref{def:noncompressive-case}), the smooth solution is global in time.
One key point of the proof of Chen et al. is the derivation of upper and, more importantly, lower bounds (in the case of a polytropic gas $p(\rho)= \kappa \rho^\gamma$ with $\gamma \in (1,3)$) for the density $\rho$.
The upper bound is easily obtained by using the \emph{Riemann invariants} of the system.
The (time-dependent) lower bound is more subtle and relies on the control of the gradients of the Riemann invariants. 

The analysis of the singular system is more delicate in our case, where the tracking of the small parameter $\ep$ is a fundamental issue for dealing with the singular pressure $p_\ep(\rho)$ in \eqref{eq:sing-press-Euler}.
Taking inspiration from \cite{chen2017}, in this context the control of the Riemann invariants of the system allows us to provide a detailed description of the life-span of the solution, highlighting and making a distinction between the gradients blows up and the vanishing parameter $\ep$ as responsible for the breakdown of the smooth solutions. 
This last point confirms the above-mentioned numerical study of Bresch and Renardy~\cite{bresch2017}.
In the end, we perform the limit as $\ep \to 0$, so rigorously justifying the connection between \eqref{eq:sing-Euler} and  \eqref{eq:lim-Euler} for ``well-prepared'' initial data.
This convergence result is likely the main novelty of the present paper, where, to the best of our knowledge, the limit from the soft congestion model to the free-congested Euler equations is proven for the first time.


\bigskip
A common point of the strategies dealing with the two classes of solutions to \eqref{eq:sing-Euler} studied in this paper, i.e. global-in-time weak bounded solutions and local-in-time $\mathcal{C}^1$ solutions, is the use of the Riemann invariants.
Their control gives indeed a refined estimate for the density of system \eqref{eq:sing-Euler} with the singular pressure law $p_\ep(\rho)$ in \eqref{eq:sing-press-Euler} in the smooth setting, in the regions close to the maximal congestion constraint.
On the other hand, in the context of weak solutions,
an $\ep$-uniform bound in $L^\infty$ of the (singular) internal energy $H_\ep$ satisfying $H'_\ep(\rho)\rho -H_\ep(\rho) = p_\ep(\rho)$ follows as an application of the invariant region method.
We point out that the key assumption for providing the internal energy bound in the weak framework
is exactly the same hypothesis from which the control of the Riemann invariants of the p-system in the smooth setting (see Remark~\ref{rmk:initial} below) is obtained.

Of course other kinds of solutions to the compressible Euler system available in literature would be appealing for our problem with a hard-sphere potential, but they are out of the scope of the present paper.
We just mention the finite-energy solutions studied by LeFloch and Westdickenberg in~ \cite{lefloch2007} and by Chen and Perepelitsa in \cite{chen2015}.
In that context, the bound on $H_\ep(\rho)$ would hold in $L^\infty_tL^1_x$ rather than $L^\infty_{t,x}$.
Lastly, the case of $BV$ solutions, displaying a quite vast literature, see for instance the books of Bressan~\cite{bressan2000}, Dafermos [Chapter 15, \cite{dafermos2000}] and references therein, will be addressed in a forthcoming investigation specifically devoted to that issue.


\paragraph{Notations and conventions}
\begin{itemize}
    \item Given a Banach space $\mathcal{B}$, we indistinctly use both $\mathcal{B}([a, b]\times \Omega)$ and $\mathcal{B}([a,b];\mathcal{B}(\Omega))$, where $[a, b] \subset \mathbb{R}^+$ and $\Omega \subset \mathbb{R}$ (thus in the second part on the smooth setting we often shortly denote $C^1_{t,x}=C^1([0,T]\times \mathbb{R})$). In the case where the time and space functional spaces $\mathcal{B}_1, \mathcal{B}_2$ are different one from another, we use the standard notation $\mathcal{B}_1([a,b]; \mathcal{B}_2(\Omega))$. 
    \item We use the notation $f_1 \lesssim f_2$ if there exists a constant $C$, independent of $\ep$, such that $f_1 \le C f_2$. 
    We also employ the notation $f(t,x)=\mathcal{O}(\ep^\alpha)$, with the constant $\alpha \in \mathbb{R}$, which means that $f(t,x)=\ep^\alpha\tilde{f}(t,x),$ where $\tilde{f}$ is a bounded continuous function in time and space.
\end{itemize}

\subsection*{Main results}
We first present our result for global weak solutions to \eqref{eq:sing-Euler}, then we discuss the case of regular solutions.

\paragraph{Existence of global weak (bounded) solutions.}
We initially assume that
\begin{equation}\label{eq:initial-data}
(\rho_\varepsilon^0, m_\varepsilon^0) \in \big(L^\infty(\R) \big)^2
\end{equation}
where, for some $C_0 > 0$,
\begin{equation}\label{eq:alpha0-ep}
0 \leq \rho^0_\ep \leq 1- C_0 \ep^{\frac{1}{\gamma-1}} =: A^0_\ep  \quad \text{a.e. on} ~~ \R,
\end{equation}
\begin{equation}\label{eq:m0-ep}
m^0_\ep (x) \leq A^0_\ep \rho^0_\ep(x)\quad \text{a.e. on} ~~\R.
\end{equation}

\bigskip
\begin{df}[Weak entropy solutions to \eqref{eq:sing-Euler}]{\label{df:sol-sing-Euler}}
	Let $(\rho_\varepsilon^0(x), m_\varepsilon^0(x))$ satisfying \eqref{eq:initial-data}-\eqref{eq:m0-ep}. 
	We call $(\rho_\ep,m_\ep)$ \emph{global weak entropy solution} to \eqref{eq:sing-Euler} if the following hold:
	\begin{itemize}
		\item $(\rho_\ep,m_\ep) \in (L^\infty(\R_+ \times \R))^2$ and there exists $A_\ep > 0$ such that
		\[
		0 \leq \rho_\ep \leq A_\ep < 1 ~ \quad \text{and} \quad |m_\ep| \leq A_\ep \rho_\ep
		\quad \text{a.e.};
		\]
		\item the mass and momentum equations are satisfied in the weak sense
		\[
		\int_{\R_+} \int_{\R} \rho_\ep \partial_t \varphi \ \d x \d t 
		+ \int_{\R_+} \int_{\R} m_\ep \partial_x \varphi \ \d x \d t
		= - \int_{\R} \rho_\ep^0(x) \varphi(0,x) \d x \quad \forall \ \varphi \in \mathcal{C}^\infty_c(\R_+ \times \R);
		\]
		\begin{align*}
		& \int_{\R_+} \int_{\R} m_\ep \partial_t \varphi \ \d x \d t 
		+ \int_{\R_+} \int_{\R} \left(\frac{m_\ep^2}{\rho_\ep} + p_\ep(\rho_\ep) \right) \partial_x \varphi \ \d x \d t \\
		&\quad 	= - \int_{\R} m_\ep^0(x) \varphi(0,x) \d x \quad \forall \ \varphi \in \mathcal{C}^\infty_c(\R_+ \times \R),
		\end{align*}
		\item the entropy inequality is satisfied, i.e. for any pair $(\eta,q)$ of entropy-entropy flux, $\eta$ convex, and $\phi \in \mathcal{C}^\infty_c(\R_+ \times \R)$, $\phi \geq 0$,
		\[
		-\int_{\R_+} \int_{\R} \eta(\rho_\ep,m_\ep) \partial_t \phi \ \d x \d t 
		- \int_{\R_+} \int_{\R}q(\rho_\ep, m_\ep) \partial_x \phi \ \d x \d t
		\leq 0 .
		\]
	\end{itemize}
\end{df}

\bigskip
\begin{thm}[Existence of global weak solutions]\label{thm:weak-sol-ep}
	Consider the pressure law \eqref{eq:sing-press-Euler} with $\kappa=0$ and $\gamma \in (1,3]$. 
	Let $(\rho^0_\ep,m^0_\ep)$ satisfy conditions \eqref{eq:initial-data}-\eqref{eq:m0-ep}. 
	Then there exists a global weak entropy solution $(\rho_\ep, m_\ep)$ to \eqref{eq:sing-Euler} in the sense of Definition \ref{df:sol-sing-Euler}.
	Moreover, the following inequality holds
	\begin{equation}\label{eq:bound-weak-rho}
	0 \leq \rho_\ep \leq 1-C\ep^{\frac{1}{\gamma-1}} \quad \text{a.e.}
	\end{equation}
	for some generic constant $C$ independent of $\ep$. 
\end{thm}

\bigskip
\begin{rmk}
As already said in the introduction, we believe that this weak existence result should extend to more general singular pressure laws \eqref{eq:description-pressure}, such that  $p_\ep (\rho)$ essentially behaves like $\rho^\gamma$ with $\gamma \in (1,3]$ for small $\rho$ (the precise hypotheses close to $0$ can be found in ~\cite{chen2003}), while 
\[
\lim_{\rho \to 1^-} (1-\rho)^\beta p_\ep(\rho) > 0 \quad \text{for some}~ \beta > 1
\]
in the vicinity of the congested region. The reasoning of Chen and LeFloch in~\cite{chen2003} should apply to this case, even though they have not explicitly
dealt with hard sphere potential pressure laws.
We propose here an alternative and explicit proof, where the compensated compactness argument employs smart combinations of a finite number of entropy-entropy flux pairs,
rather than requiring the analysis of the (integral) entropy kernel.
In return, our method requires the specific pressure law~\eqref{eq:sing-press-Euler} with $\kappa = 0$.\\
It is also worth pointing out that our compensated compactness arguments cannot be used again when dealing with the limit $\ep \to 0$, the main stumbling block being the lack of estimates on the singular pressure $p_\ep(\rho_\ep)$.
\end{rmk}

\paragraph{Existence of smooth solutions and asymptotic behavior as $\ep \to 0$.}
The second part of the paper is dedicated to the investigation on regular ($C^1_{t,x}$) solutions to system~\eqref{eq:sing-Euler}.
In this setting, we show that the passage to Lagrangian coordinates allows us to provide a refined description of the solutions. 
More precisely, we are able to analyse and exactly quantify the influence of the singular component of the pressure ($p_{1,\ep}$ in \eqref{eq:sing-press-Euler}) on the breakdown of the smooth solutions. 
After obtaining an existence theory at $\ep$ fixed, we are finally allowed to justify the asymptotics $\ep \rightarrow 0$ 
under additional assumptions on the initial data close to the congestion constraint.

\medskip

{\it Passage to Lagrangian coordinates.}
The previous system \eqref{eq:sing-Euler} is written in the so-called Eulerian coordinates $(t,x)$.
If, instead of $x$, we choose as space variable the material coordinate $\tilde{x}$ such that
\[
\d x = u \d t + v \d \tilde{x} \quad \text{where} \quad v := \frac{1}{\rho},
\]
then the system can be rewritten as
\begin{equation}\label{eq:p-system}
\begin{cases}
\partial_t v - \partial_{\tilde{x}} u =0, \\
\partial_t u +  \partial_{\tilde{x}} \tp_\ep (v) = 0,
\end{cases}
\quad \text{for} \quad (t,\tilde{x}) \in \R_+ \times \R,
\end{equation}
with the pressure law $\tp_\ep(v) := p_\ep(v^{-1})$.
For sake of simplicity, when it is clear that we are in the Lagrangian setting, we shall drop hereafter the notation $\tilde{\cdot}$.
In the context of gas motion, the variable $v$ denotes the \emph{specific volume} (the reciprocal of the gas density) and system \eqref{eq:p-system} is called \emph{p-system}. 
The change of variable can be justified not only for smooth solutions but also in the framework of weak bounded solutions, as shown by Wagner in~\cite{wagner1987}.
Nevertheless, in the latter setting, the definition of weak solutions for the Lagrangian equations must be adapted in the  regions where vacuum occurs. 
This discussion is detailed in \cite{wagner1987} and [Section 1.2, \cite{serre1999}]. 
\\
As $\ep \to 0$, we expect that the sequence of solutions $(v_\ep,u_\ep)_\ep$ to \eqref{eq:p-system} converges to a solution $(v,u)$ of the following free-congested p-system (namely the Lagrangian version of \eqref{eq:lim-Euler}):
\begin{subnumcases}{\label{eq:lim-p}}
\partial_t v - \partial_x u = 0 \\
\partial_t u + \partial_x p + \kappa \partial_x v^{-\tgamma} = 0 \\
v \geq 1, ~ (v-1) p = 0, ~ p \geq 0
\end{subnumcases}

\bigskip
{\it Statement of the results.}
This part provides two main results. The first one concerns the existence of smooth solutions at $\ep>0$ fixed, and makes a distinction between two cases
which depend on the initial data and are defined below.
\begin{df}\label{def:noncompressive-case}
Let us introduce the function $\theta_\ep$, defined as
\[
\theta_\ep(v) := \int_{v}^{+\infty} \sqrt{-p'_\ep(\tau)} \d \tau
\]
and the Riemann invariants
\[
w^0_\ep := u_\ep^0 + \theta_\ep(v^0_\ep), \quad 
z^0_\ep := u^0_\ep - \theta_\ep(v^0_\ep).
\]
At a point $x \in \R$, the initial datum $(v^0_\ep,u^0_\ep)$ is said to be \emph{rarefactive} if it is such that
\begin{equation}\label{hyp:init-compr}
\partial_x w^0_\ep(x) \geq 0 \quad \text{and} \quad \partial_x z^0_\ep(x) \geq 0
\end{equation}
and \emph{compressive} otherwise.
\end{df}

In this regular setting, we assume that the initial datum is $\mathcal{C}^1$ and that
\begin{align}\label{initial-bound}
    \|(v_\ep^0, u_\ep^0)\|_{L^\infty(\mathbb{R})} 
    + \|(\partial_x v_\ep^0, \partial_x u_\ep^0)\|_{L^\infty(\mathbb{R})}
    \le C
\end{align}
for some generic positive constant $C$ independent of $\ep$.
As in the weak setting, we also assume~\eqref{eq:alpha0-ep} for the initial density.  
This is equivalent at assuming
\begin{equation}\label{hyp:init-data}
(v^0_\ep-1)^{\gamma-1} \geq  C\ep
\end{equation}
for some $C>0$, on the initial specific volume.
Furthermore, our initial data $(u_\ep^0, v_\ep^0)$ satisfy

\begin{equation}\label{hyp:init-data-2}
\sqrt{c_\ep^0} \ \partial_x w_\ep^0 \leq Y^0 , \qquad \sqrt{c_\ep^0} \ \partial_x z_\ep^0 \leq Q^0,
\end{equation}
for some positive constants $Y^0, Q^0$ independent of $\ep$. 
The meaning of that will be clarified in Section \ref{sub-initial-data}. 

\medskip
We state our existence result in the smooth setting.

\begin{thm}[Existence and life-span of $(v_\ep,u_\ep)$]{\label{thm:ex-ep}}
Let 
\[
p_\ep(v) = \dfrac{\ep}{(v-1)^\gamma} + \dfrac{\kappa}{v^{\tgamma}} \quad \text{with} \quad 
\kappa > 0, ~ \gamma > 1, ~\tgamma \in (1,3) ~\text{and}~ \ep \leq \ep_0 ~\text{small enough}. 
\]
Assume that the initial data satisfy \eqref{initial-bound}-\eqref{hyp:init-data}-\eqref{hyp:init-data-2}. We have two cases.
\begin{enumerate}
	\item If the initial datum is \emph{everywhere rarefactive} in the sense of Definition \ref{def:noncompressive-case}, 
	then there exists a unique global-in-time $\mathcal{C}^1_{t,x}$ solution $(v_\ep,u_\ep)$, whose $\mathcal{C}^1_{t,x}$-norm is independent of $\ep$.
	\item{\label{it:compres}} Otherwise (i.e. if there exists $x^* \in \R$ such that $\partial_x w^0_\ep(x^*) < 0$ \text{or} $\partial_x z^0_\ep(x^*) < 0$), there exists a unique local $\mathcal{C}^1_{t,x}$ solution $(v_\ep,u_\ep)$ which breaks down in finite time.
\end{enumerate}
Moreover, in case \ref{it:compres} where a blowup in finite time occurs, we have the following lower bounds on the maximal time of existence $T^*_\ep < +\infty$
    \begin{equation}\label{bound_tep}
 T_\ep^* 
 \geq \begin{cases}\displaystyle 
  \inf_{x^* \in \R} \dfrac{\ep^{\frac{1}{2(\gamma-1)}}}{C \sqrt{c_\ep^0(x^*)} \max\{-\partial_x w^0_\ep(x^*),-\partial_x z^0_\ep(x^*) \}} \quad & \text{if} \quad \gamma \in (1,3), \\
 \displaystyle  \inf_{x^* \in \R} \dfrac{\ep^{\frac{1}{4}}}{C \sqrt{c_\ep^0(x^*)} \max\{-\partial_x w^0_\ep(x^*),-\partial_x z^0_\ep(x^*) \}} \quad & \text{if} \quad \gamma = 3, \\
  \displaystyle  \inf_{x^* \in \R} \dfrac{\ep^{\frac{1}{\gamma+1}}}{C \sqrt{c_\ep^0(x^*)} \max\{-\partial_x w^0_\ep(x^*),-\partial_x z^0_\ep(x^*) \}} \quad & \text{if} \quad \gamma > 3,
 \end{cases} 
 \end{equation}
 where $C>0$ is a suitable constant independent of $\ep$.
\end{thm}

Notice that, in full generality, the maximal existence time $T^*_\ep$ depends on $\ep$ and may a priori degenerate to $0$ if no additional assumption is satisfied by the initial data $(\partial_x w^0_\ep,\partial_x z^0_\ep)$.
The specific hypotheses that ensure an $\ep$-uniform lower bound on $T^*_\ep$ are given later on in Assumption \ref{ass-IV-setup}.
The derivation of this lower bound is the preliminary step for the analysis of the singular limit $\ep \to 0$.
Before stating our convergence result, let us recall the notion of solutions for the target limit system \eqref{eq:lim-p}.

\begin{df}[Weak solutions to the free-congested p-system]
{\label{df:sol-p-system}}
	Let $(v^0, u^0) \in \mathcal{C}^1(\R)$ satisfying 
	\[
	v^0 (x) \geq 1 \quad \forall \ x \in \R,
	\]
	and let $T>0$ be fixed.
	We say that $(v,u,p)$ is a \emph{weak solution} to \eqref{eq:lim-p} on the time interval $[0,T]$ if the following hold:
	\begin{itemize}
		\item the mass equation is satisfied a.e.
		\[
		\partial_t v - \partial_x u = 0, \qquad v_{|t=0} = v^0.
		\]
		\item the momentum equation is satisfied in the sense of distributions
 		\begin{align*}
 		\int_{\R_+} \int_{\R} u \partial_t \varphi \ \d x \d t 
 		+ \kappa \int_{\R_+} \int_{\R} v^{\tgamma} \ \partial_x \varphi \ \d x \d t
 	    + \int_{\R_+} \int_{\R} \partial_x \varphi \ \d p(t,x)  \\
 		= - \int_{\R} u^0(x) \varphi(0,x) \d x \quad \forall \ \varphi \in \mathcal{C}^\infty_c(\R_+ \times \R);
 		\end{align*}
        \item the congestion and exclusion constraints are satisfied in the following sense
        \[
        v(t,x) \geq 1 \quad \forall \ (t,x) \quad \text{and} \quad
        p \geq 0, ~ (v-1)p = 0 \quad \text{in}~ \mathcal{D}'.
        \]
	\end{itemize}
\end{df}

\medskip
The result below establishes the validity of the soft congestion approximation to the free-congested Euler equations. 

\begin{thm}[Singular limit in the smooth setting]\label{thm:main-smooth}
Under the hypotheses of the previous Theorem and suitable additional assumptions on the initial data $(v_\ep^0, u_\ep^0) \in \mathcal{C}^1(\mathbb{R})$ (see Assumptions \ref{ass-IV-setup}-\ref{ass-II-setup} in Section \ref{sub-initial-data}), there exist a time interval $[0,T]$, where $T>0$ is independent of $\ep$, a limit initial datum $(v^0,u^0)$ and a triplet $(v, u, p)$ such that the following convergences hold (up to the extraction of a subsequence):
\begin{align*}
v_\ep^0 \rightarrow v^0 \quad & \text{strongly in}~ \mathcal{C}([-L,L]) ~ \text{and weakly-* in} ~ W^{1,\infty}(\R)\\
u_\ep^0 \rightarrow u^0 \quad & \text{strongly in}~ \mathcal{C}([-L,L])~ \text{and weakly-* in} ~ W^{1,\infty}(\R)\\
 v_\ep \rightarrow v \quad & \text{strongly in}~  \mathcal{C}([0,T]\times [-L,L])  ~\text{and weakly-* in} ~ L^\infty((0,T);W^{1,\infty}(\R)),\\
 u_\ep \rightarrow u \quad & \text{strongly in} ~  L^q((0,T); \mathcal{C}([-L,L])), \quad \forall  \ q \in [1, +\infty), \ L > 0,\\  & \quad  \text{and weakly-* in} ~ L^\infty((0,T);W^{1,\infty}(\R)) \\
 p_{\ep,1}(v_\ep) \rightharpoonup p \quad  & \text{in} \quad \mathcal{M}_+((0,T) \times (-L,L)) \quad \forall \ L >0.
 \end{align*}
 Moreover, the limit $(v,u,p)$ is a weak solution of the free-congested p-system associated to the initial datum $(v^0,u^0)$ in the sense of Definition \ref{df:sol-p-system}.
 Finally, the couple $(v,u)$ satisfies the incompressibility constraint in the congested domain, i.e.
 \[
 \partial_x u = 0 \quad \text{a.e. on} \quad \{v=1\}.
 \]
\end{thm}

\medskip
\begin{rmk}\label{rmk-assumption-thm7}
For sake of clarity, we have postponed to Section \ref{sub-initial-data} the precise statement of the two additional assumptions which are needed to pass to the limit as $\ep \to 0$. 
Assumption \ref{ass-IV-setup} ensures that the whole sequence $(v_\ep,u_\ep)_\ep$ exists on a time interval $[0,T]$ independent of $\ep$,
while Assumption \ref{ass-II-setup} is a technical hypothesis which basically states that the initial specific volume $v^0$ is not congested (i.e. equal to $1$) in the whole domain. 
As in previous studies dealing with the same singular limit (see for instance~\cite{peza2015}), this latter assumption is required to control the pressure $p_\ep(v_\ep)$ in an appropriate functional space.
Regarding Assumption \ref{ass-IV-setup}, the idea is to control (in terms of $\ep$) $\partial_x u_\ep^0, \, \partial_x v_\ep^0$ in the regions initially close to the congestion constraint. 
\end{rmk}

\medskip

\begin{rmk} [Assumptions on the pressure]
 \label{rmk-tgamma}
 We have assumed in Theorem~\ref{thm:ex-ep} and Theorem~\ref{thm:main-smooth} that the exponent $\tgamma$ of the non-singular component of the pressure $p_2$ lies in the interval $(1,3)$. 
 This assumption is mainly used when deriving a lower bound on the sequence of the maximal times $(T^*_\ep)_\ep$ (see Proposition~\ref{prop:inf-T*}). 
 However it is actually not necessary to guarantee the first part of Theorem~\ref{thm:ex-ep}, that is the global existence or the blow-up in finite time depending on the presence or not of a compression in the initial datum.\\
The specific form of the pressure \eqref{eq:sing-press-Euler} (which blows
up close to 1 like a power law) is used in Sections \ref{sec:strong-sol-ep}-\ref{sec:strong-sol-lim} to exhibit the small scales associated to the singular limit $\ep \to 0$ (see in particular estimate \eqref{bound_tep} and Assumption \ref{ass-IV-setup}). 
Nevertheless, we expect similar results for
more general hard-sphere potentials. 
All the estimates will then depend on the specific balance between the parameter $\ep$ and the type
of the singularity close to $1$ encoded in the pressure law.
\end{rmk}

\medskip
\begin{rmk} [Control of the Riemann invariants and link with the internal energy] \label{rmk:initial}
The crucial assumption that we make in both, weak and smooth, settings is \eqref{eq:alpha0-ep} (reformulated in \eqref{hyp:init-data} with the specific volume).
It bounds from below the distance between the initial density $\rho^0_\ep$ and the maximal threshold $\rho^* = 1$, and allows a control of the Riemann invariants, $w_\ep$ and $z_\ep$ (see respectively Section~ \ref{sec:inv-weak} and Section \ref{sec:inv-smooth}).
From another perspective, the initial condition~\eqref{eq:alpha0-ep} guarantees that the internal energy at time $0$ (and consequently for all times) is bounded uniformly with respect to $\ep$.
	Indeed, defining the internal energy as
	\begin{equation}\label{df:energy}
	H_\ep(\rho) = \dfrac{\ep}{\gamma-1} \dfrac{\rho^\gamma}{(1-\rho)^{\gamma-1}}
	\end{equation}
	which is such that
	\[
	\rho H'_\ep(\rho) - H_\ep(\rho) = p_\ep(\rho),
	\]
	we ensure, thanks to \eqref{eq:alpha0-ep}, that
	\begin{equation} \label{eq:bound-H0}
	H_\ep(\rho^0_{\ep}) 
	= \dfrac{\ep}{\gamma-1} \dfrac{(\rho^0_{\ep})^\gamma}{(1-\rho^0_{\ep})^{\gamma-1}}
	\leq \dfrac{1}{(\gamma-1)C_0^{\gamma-1}}
	=: H_0.
	\end{equation}
	This provides a connection between the weak ($L^\infty$) and the smooth setting which are both investigated in this work. 
	
\end{rmk}

\subsection*{Organisation of the paper}
Section \ref{sec:existence-weak} is devoted to the existence of global weak $L^\infty$ solutions to \eqref{eq:sing-Euler} at $\ep > 0$ fixed.
Next, in Sections \ref{sec:strong-sol-ep} and \ref{sec:strong-sol-lim} we address the issue of smooth solutions, first analysing the singularity formation problem at $\ep$ fixed. Then, we let $\ep \to 0$ to recover the free-congested system \eqref{eq:lim-Euler} in the limit.
We postponed to the Appendix the proof of technical results. 


\section{Existence of global weak solutions at $\ep$ fixed}
{\label{sec:existence-weak}}

The aim of this section is to prove the existence of global weak solutions to \eqref{eq:sing-Euler} by passing to the limit $\mu \to 0$ in the following regularized system
\begin{subnumcases}{\label{eq:viscous}}
\partial_t \rho_{\mu} + \partial_x m_{\mu} = \mu \partial_{xx} \rho_{\mu}, \\
\partial_t m_{\mu} + \partial_x\left(\frac{m_{ \mu}^2}{\rho_{\mu}}\right) + \partial_x p_\ep(\rho_{\mu}) = \mu \partial_{xx} m_{\mu}
\end{subnumcases}
where
\begin{equation}\label{df:pressure-weak}
p_\ep(\rho) 
= \ep \left(\dfrac{\rho}{1-\rho}\right)^\gamma \quad \text{with} \quad \gamma \in (1,3].
\end{equation}
Initially, we define 
\begin{equation}\label{initial-data-reg}
(\rho^0_{\ep, \mu}(x), m^0_{\ep, \mu}(x)) :=(\rho^0_\varepsilon (x), m^0_\varepsilon (x)) \star j_\mu (x),
\end{equation}
where $j_\mu (x)$ is a standard mollifier, in such a way that
\begin{equation*}\label{eq:rho0-mu}
0 < a_\mu^0 \leq \rho^0_{\ep, \mu} \leq A^0_\ep <1 , \quad m^0_{\ep, \mu} (x) \leq A^0_\ep \rho^0_{\ep, \mu}(x) \quad \text{a.e. on} ~ \R.
\end{equation*}
In view of later purposes, we rewrite the previous system in the compact form
\begin{equation} \label{eq:sys-v-mu}
\begin{cases}
\partial_t U_{\mu} + \partial_x f_\ep(U_{\mu}) = \mu \partial_{xx} U_{\mu},\\
U_{\mu}(0, x):=(\rho_{\ep, \mu}^0, m_{\ep, \mu}^0),
\end{cases}
\end{equation}
with
\[
U_{\mu}=\begin{pmatrix}\rho_{ \mu} \\ m_{ \mu}\end{pmatrix},
\quad f_\ep(U_{\mu}) = \begin{pmatrix} m_{\mu} \\ \dfrac{m_{\mu}^2}{\rho_{\mu}} + p_\varepsilon(\rho_{\mu})
\end{pmatrix}.
\]
In the next Subsection \ref{sec:inv-weak}, we prove the existence and uniqueness of global solutions to~ \eqref{eq:viscous} deriving at the same time uniform bounds with respect to the viscosity parameter $\mu$.
Then, we show in Subsection \ref{sec:mu-to-0} how we can use these bounds to pass to the limit $\mu \to 0$.

Note that we have dropped the subscript $\ep$ in the above variables/equations (except in the pressure law and the initial datum where it is important to keep in mind that $\ep > 0$) to lighten the notation, the limit $\ep \to 0$ being not considered in this section.

\subsection{Analysis  of the viscous regularized system} {\label{sec:inv-weak}}

\begin{prop}[Invariant region]\label{prop:invariant}
	Define the Riemann invariants of system \eqref{eq:sing-Euler} as 
	\begin{equation}\label{df:w}
	w(\rho, m) = \dfrac{m}{\rho} + \int_0^{\rho} \dfrac{\sqrt{p'_\ep(s)}}{s} ds , 
	\end{equation}
	\begin{equation} \label{df:z}
	z(\rho, m) = \dfrac{m}{\rho} - \int_0^{\rho} \dfrac{\sqrt{p'_\ep(s)}}{s} ds .
	\end{equation}
	Under the initial conditions \eqref{eq:initial-data}-\eqref{eq:m0-ep}, the quantity $M = \|w(\rho^0_\mu,m^0_\mu)\|_{L^\infty_x}$ is bounded uniformly with respect to $\ep$ and $\mu$, and the domain
	\begin{equation}\label{df:Sigma}
	\Sigma := \left\{(\rho_\mu,m_\mu), ~ w(\rho_\mu,m_\mu) \leq M  \right\} \cap \left\{(\rho_\mu,m_\mu), ~ z(\rho_\mu,m_\mu) \geq -M  \right\}
	\end{equation}
	is an invariant region for \eqref{eq:sys-v-mu}.
\end{prop}

This result follows from standard arguments, for sake of clarity we postpone its proof to the Appendix (see Subsection~\ref{sec:inv-reg}).
Equipped with these controls of the Riemann invariants, we infer a priori $L^\infty$ bounds on the variables $(\rho_\mu,m_\mu)$, from which the existence of regular solutions follows. 

\begin{prop}\label{prop:bounds-existence}
	Let $\mu, \ \ep > 0$ be fixed and assume the initial conditions \eqref{eq:initial-data}-\eqref{eq:m0-ep}.
	Then, for any $T >0$, there exists a unique global smooth solution $U_{\mu}= (\rho_{\mu},m_{\mu})$ satisfying:
	\begin{equation}
	0 < a_\mu \leq \rho_{\mu} \leq A_\ep < 1, \quad |m_{\mu}| \leq B \, \rho_{\mu} 
	\quad \text{a.e.},
	\end{equation}
	for some constant $a_\mu> 0$, while $A_\ep > 0$ is independent of $\mu$, and $B$ is independent of $\mu$ and $\ep$.
	Moreover there exists a constant $C > 0$ independent of $\mu$ and $\ep$, such that
 	\begin{equation}
    \rho_\mu \leq 1 - C \ep^{\frac{1}{\gamma-1}} \quad \text{a.e.}.
	\end{equation}
\end{prop}

\bigskip
\begin{proof}

	Thanks to Proposition \ref{prop:invariant}, we deduce that
	\begin{equation}\label{eq:bound-G}
	2 M
	\geq \|w-z\|_{L^\infty}
	= 2\Bigg\| \int_0^{\rho_{\mu}} \dfrac{\sqrt{p'_\ep(s)}}{s} ds \Bigg\|_{L^\infty}
	=  2\| \Theta_\ep(\rho_{\mu})\|_{L^\infty}
	\end{equation}
	where $\Theta_\ep$ denotes the primitive of $s\mapsto \sqrt{p'_\ep(s)}/s$ vanishing at $0$.
	From \eqref{eq:bound-G} (recalling that $\gamma >1$) we then infer that there exists a generic constant $C > 0$, independent of $\ep$ and $\mu$, such that
	\[
	\dfrac{\sqrt{\ep}}{(1-\rho_\mu)^{\frac{\gamma-1}{2}}} \mathbf{1}_{\{\rho_\mu \geq 1/2 \}}
	\leq 
	\Theta_\ep(\rho_{\mu})\mathbf{1}_{\{\rho_\mu \geq 1/2 \}}
	\leq C M,
	\]
	that is
	\[
	\rho_{\mu} \leq 1 - C \ep^{\frac{1}{\gamma-1}} =: A_\ep \quad \text{a.e.}
	\]
	On the other hand, we have
	\[
	-\rho_{\mu} M + \rho_{\mu} \int_0^{\rho_{\mu}} \dfrac{\sqrt{p'_\ep(s)}}{s} ds \le m_{\mu} 
	\le \rho_{\mu} M + \rho_{\mu}\int_0^{\rho_{\mu}} \dfrac{\sqrt{p'_\ep(s)}}{s} ds.
	\]
	and thus 
	\[
	|m_{\mu}| \leq 2M \rho_{\mu}.
	\]
	Since $u_{\mu} := \frac{m_{\mu}}{\rho_{\mu}}$ is uniformly bounded, we can obtain the lower bound on $\rho_{\mu}$ (which possibly vanishes as $\mu \rightarrow 0$) and the global existence of the solution $(\rho_{\mu},m_{\mu})$ as well, thanks to classical results on parabolic systems. 
	The interested reader is referred to [Theorem 1.0.2 ,\cite{lu2002}].
\end{proof}

\subsection{Vanishing viscosity limit} {\label{sec:mu-to-0}}

We are now ready to deal with the limit $\mu \to 0$, splitting the reasoning in two steps. First, we exhibit four pairs of entropy-entropy flux $(\eta_i,q_i)$, so providing $H^{-1}_{t,x}$-compactness (in time and space) for each one. Later, we prove the strong convergence of the couple $(\rho_{\mu},m_{\mu})$. The latter result heavily relies on the following classical compensated compactness theorem due to Murat and Tartar (see for instance Dafermos book~[Section 16.2, \cite{dafermos2000}]).

\begin{thm}[Div-Curl Lemma]\label{lem:div-curl}
	Let $\Omega$ be an open subset of $\R^m$, $m \geq 2$, $(G_j)_j, \, (H_j)_j$ be sequences of vector fields belonging to $(L^2(\Omega))^m$ such that
	\[
	G_j \rightharpoonup \bar{G}, \quad H_j \rightharpoonup \bar{H} \quad \text{weakly in} \quad (L^2(\Omega))^m.
	\]
 	If $\displaystyle (div \, G_j)_j, \, (curl \, H_j)_j \subset \text{compact sets of  } H_{loc}^{-1}(\Omega),$
 	then 
 	$$G_j \cdot H_j \rightharpoonup \bar{G} \cdot \bar{H}
 	\quad \text{weakly in} \quad \mathcal{D}'.$$
\end{thm}

\subsubsection{Entropy-entropy flux pairs, compactness in $H^{-1}$}

One can check that the pairs $(\eta_i, q_i)= (\eta_i(\rho,m), q_i(\rho,m))$, $i=1, \dots, 4$, defined as follows
\begin{align*}
(\eta_1,q_1) & = (\rho \ , \ m), \\
(\eta_2,q_2) & = \left(m \ , \ \frac{m^2}{\rho} + p_\ep(\rho) \right), \\
(\eta_3,q_3) & = \left(
\frac{m^2}{2\rho} + \rho \int^\rho{\frac{p_\ep(s)}{s^2} \d s} \ ,
\ \frac{m^3}{2\rho^2} + m \left[\frac{p_\ep(\rho)}{\rho}+ \int^\rho{\frac{p_\ep(s)}{s^2} \d s}\right]  \right), \\
(\eta_4,q_4) & = \Bigg(
\frac{m^3}{\rho^2} + 6m \int^\rho\frac{p_\ep(s)}{s^2} \d s , \\
& \qquad	\frac{m^4}{\rho^3} + 3m^2 \left[\frac{p_\ep(\rho)}{\rho^2} + \frac{2}{\rho}\int^\rho\frac{p_\ep(s)}{s^2} \d s \right] 
+ 6 \left[p_\ep(\rho)\int^\rho\frac{p_\ep(s)}{s^2} \d s - \int^\rho\frac{(p_\ep(s))^2}{s^2} \d s \right]
\Bigg) ,
\end{align*}
are entropy-entropy flux pairs for system \eqref{eq:sing-Euler}. Namely, by definition (see again \cite{dafermos2000}), given a smooth solution $U=(\rho,m)$ to system \eqref{eq:sing-Euler}, one has that 
\[
\partial_t \eta_i(U) + \partial_x q_i(U) = 0, \quad i=1, \dots,4.
\]
The pairs $(\eta_1,q_1)$ and $(\eta_2,q_2)$ are associated with the mass and momentum equation respectively, while $(\eta_3,q_3)$ corresponds to the energy equality.\\
Thanks to the $L^\infty$ bounds on $U_\mu = (\rho_{\mu},m_{\mu})$, one can easily check that the following lemma holds.

\begin{lemma}
	There exists a positive constant $C$ independent of $\mu$ such that $\forall \ i=1,\dots 4$:
	\begin{equation}\label{eq:bound-ei-qi}
	\|\eta_i (U_\mu)\|_{L^\infty_{t,x}} + \|q_i (U_{\mu})\|_{L^\infty_{t,x}} \leq C,
	\end{equation}
	and
	\begin{equation}\label{eq:bound-d-ei-qi}
	\|\partial_\rho \eta_i (U_{\mu})\|_{L^\infty_{t,x}} +\|\partial_m \eta_i (U_{\mu})\|_{L^\infty_{t,x}} \leq C.
	\end{equation}
\end{lemma}

\bigskip
The main result of this subsection is stated in the next proposition. 

\begin{prop}\label{prop:compact-entropy}
	The following property holds for $i=1,\dots,4$:
	\begin{equation}
	\partial_t \eta_i(U_\mu ) + \partial_x q_i(U_\mu ) 
	\ \subset ~\text{compact set of}~ H^{-1}_{loc}(\mathbb{R}^+ \times \mathbb{R}).
	\end{equation}
\end{prop}

The proof of this proposition relies on the following equation obtained by formally multiplying system \eqref{eq:sys-v-mu} by $\nabla \eta_i(U_\mu) =(\partial_{\rho}\eta_i(U_\mu), \partial_{m}\eta_i(U_\mu))$:
\begin{align} \label{eq:split-eq-eta}
& \partial_t  \eta_i(U_\mu) + \partial_x q_i(U_\mu) \nonumber \\
& = \mu \partial_{xx} \eta_i(U_\mu)-\mu \partial_x U_\mu . \nabla^2 \eta_i(U_\mu) . \partial_x U_\mu ^T
\end{align}
and this lemma below, whose proof can be found in [Lemma 16.2.2,\cite{dafermos2000}].

\begin{lemma}[]\label{lem:dafermos}
	Let $\Omega$ be an open subset of $\R^m$ and $(\phi_j)_j$ a bounded sequence in $W^{-1,p}(\Omega)$ for some $p > 2$.
	Furthermore, let $\phi_j = \chi_j + \psi_j$, where $(\chi_j)_j$ lies in a compact set of $H^{-1}(\Omega)$, while $(\psi_j)_j$ lies in a bounded set of the space of measures $\mathcal{M}(\Omega)$. 
	Then $(\phi_j)_j$ lies in a compact set of $H^{-1}(\Omega)$.
\end{lemma}

Hence, we need to control the two terms of the right-hand side of \eqref{eq:split-eq-eta} in $H^{-1}(\Omega)$ and $\mathcal{M}(\Omega)$ respectively.
Let us begin with the control of the second term. 
\begin{lemma}\label{lem:control:etai}
	For $i=1,\dots 4$, the sequence
	\[
	\Big(\mu \partial_x U_\mu . \nabla^2 \eta_i(U_\mu) . \partial_xU_\mu^T\Big)_\mu \quad \text{is bounded in} \quad L^1_{\rm loc}(\R_+ \times \R).
	\]	
\end{lemma}

\begin{proof}
	Testing system (\ref{eq:sys-v-mu}) against $\nabla \eta_i(U_\mu) \phi$ with $\phi \in \mathcal{C}^\infty_c(\R_+ \times \R)$, one gets
	\begin{align*}
	& -\int_{\R^+\times \R} \eta_i(U_\mu) \phi_t \, dx \, dt  -\int_{\R^+\times \R} q_i(U_\mu) \phi_x \, dx \, dt \\
	& = \mu \int_{\R^+\times \R} \eta_i(U_\mu) \phi_{xx} \, dx \, dt - \mu \int_{\R^+\times \R} \big( \partial_x U_\mu .\nabla^2\eta_i(U_\mu) . \partial_x U_\mu^T \big) \phi \, dx \, dt.
	\end{align*}
	Now, recall that $\eta_i$ and $q_i$ are uniformly bounded in $L^\infty$ (see \eqref{eq:bound-ei-qi}) and that $\eta_i$ is convex, so that
	\begin{align*}
	0 
	&\leq  \mu \int_{\R^+\times \R}   (\partial_x U_\mu \cdot  \nabla^2\eta_i(U_\mu) \cdot \partial_x U_\mu^T) \phi \, dx \, dt \\
	&\le \int_{\R^+\times \R} |\eta_i(U_\mu) \phi_t| \, dx \, dt  +\int_{\R^+\times \R} |q_i(U_\mu) \phi_x| \, dx \, dt + \mu \int_{\R^+\times \R} |\eta_i(U_\mu) \phi_{xx}| \, dx \, dt \\
	& \leq C
	\end{align*}
	where $C>0$ is a generic constant independent of $\mu$.
\end{proof}

\bigskip
Let us now pass to the control of the first term of \eqref{eq:split-eq-eta}.
\begin{lemma}\label{lem:cvg-dxa}
	For any $T > 0$ and compact set $K \subset \R$, the following convergences hold
	\begin{equation}
	\mu^2 \int_0^T \int_K \rho_{\mu} (\partial_x u_{\mu})^2 \d x \d t\longrightarrow 0 \quad \text{as} \quad \mu \to 0,
	\end{equation}
	\begin{equation}
	\mu^2 \int_0^T \int_K (\partial_x \rho_{\mu})^2 \d x \d t\longrightarrow 0 \quad \text{as} \quad \mu \to 0.
	\end{equation}
\end{lemma}

\begin{proof}
	
	Using Lemma \ref{lem:control:etai}, we have a control of
	\begin{align}\label{eq:eta3-dvp}
	& \mu (\partial_x \rho_{\mu}, \partial_x m_{\mu}) . \nabla^2 \eta_3(\rho_{\mu}, m_{\mu}) . (\partial_x \rho_{\mu}, \partial_x m_{\mu})^T \nonumber\\
	& =\mu \rho_{\mu}^{-1} \Bigg(\dfrac{m_{\mu} \partial_x \rho_{\mu}}{\rho_{\mu}} - \partial_x m_{\mu} \Bigg)^2+\mu \dfrac{p'_\varepsilon(\rho_{\mu})}{\rho_{\mu}} (\partial_x \rho_{\mu})^2 \nonumber \\
	& = \mu \dfrac{p'_\varepsilon(\rho_{\mu})}{\rho_{\mu}} (\partial_x \rho_{\mu})^2
	+ \mu \rho_{\mu} (\partial_x u_{\mu})^2,
	\end{align}
	in $L^1_{\rm loc}(\R_+ \times \R)$.
	The control of the second term directly yields the first convergence result of the Lemma.
	This is not the case for the first term $\mu \dfrac{p'_\varepsilon(\rho_\mu)}{\rho_\mu}(\partial_x \rho_\mu)^2$ (which is controlled in $L^1_{loc}(\R^+ \times \R)$), because of the degeneracy of the pressure close to vacuum.
	Then we decompose the domain by introducing $\delta \in (0,1)$ and defining
	$$ \Omega_1=\{(t,x) \in \R_+ \times \R \, | \, \rho_{\mu}(t,x) >\delta\}.$$
	On the compact set $K_1 \subset \Omega_1$, where $\rho_{\mu}$ is far from $0$, we deduce from the expression of the pressure $p_\ep$ in \eqref{df:pressure} that
	\begin{align*}
	\delta^{\gamma-2} \mu^2 \int_{K_1} (\partial_x \rho_\mu)^2 \, dx \, dt
	\le \mu^2 \int_{K_1}\dfrac{p'_\varepsilon(\rho_{\mu})}{\rho_{\mu}} (\partial_x \rho_{\mu})^2 \, dx \, dt \
	\leq C \mu,
	\end{align*}
	and thus, $\delta > 0$ being fixed,
	\[
	\mu^2 \int_{K_1} (\partial_x \rho_\mu)^2 \, dx \, dt \longrightarrow 0  \quad \text{as}~~ \mu \to 0.
	\]
	On the complementary set $\Omega_1^c = \{(t,x) \in \R_+ \times \R \, | \, \rho_{\mu}(t,x) \leq \delta\} $, we introduce $\phi \in \mathcal{C}_c^\infty(\Omega_1^c)$ such that $0 \leq \phi \le 1$ and $\phi \equiv 1$ on the compact $K_2 \subset \Omega_1^c$. 
	Testing the mass equation of system \eqref{eq:sys-v-mu} against $2 \rho_{\mu} \phi$, one gets
	\begin{align*}
	\int_{\R^+\times \R} & 2\mu (\partial_x \rho_{\mu})^2 \phi \, dx \, dt \\
	& = \int_{\R^+\times \R} \Big(\rho_{\mu}^2 \partial_t\phi + 2 \rho_{\mu}^2 u_{\mu} \partial_x \phi + \mu \rho_{\mu}^2 \partial_{xx} \phi \Big) \, dx \, dt 
	+ \int_{\R^+\times \R}2 \rho_{\mu} u_{\mu} \partial_x \rho_{\mu} \phi \, dx \, dt \\
	& \le C\delta^2 + C\delta \Bigg(\int_{\R^+\times \R} (\partial_x \rho_{\mu})^2\phi \, dx \, dt \Bigg)^\frac{1}{2} \\
	& \leq  C\delta^2 + C \dfrac{\delta^2}{\mu} + \mu \int_{\R^+\times \R} (\partial_x \rho_{\mu})^2\phi \, dx \, dt
	\end{align*}
	where we have used the control in $L^\infty$ of the velocity $u_{\mu} = \frac{m_{\mu}}{\rho_{\mu}}$.
	Multiplying both sides of the inequality by $\mu$ and absorbing the last term of the right-hand side into the left-hand side, we end up with
	\[ 
	\mu^2 \int_{K_2}(\partial_x \rho_{\mu})^2 \, dx \, dt \leq C\delta^2,
	\]
	which achieves the proof of Lemma \ref{lem:cvg-dxa} by finally letting $\delta \to 0$.
\end{proof}

\bigskip
We can now conclude the proof of Proposition~\ref{prop:compact-entropy}.

\begin{proof}[Proof of Proposition \ref{prop:compact-entropy}]
	Coming back to \eqref{eq:split-eq-eta}, we have 
	\begin{align*}
	& \partial_t  \eta_i(U_\mu) + \partial_x q_i(U_\mu) \\
	& = \mu \partial_{xx} \eta_i(U_\mu)-\mu \partial_x U_\mu. \nabla^2 \eta_i(U_\mu) . \partial_x U_\mu ^T\\
	& =: I_1(\mu) + I_2(\mu).
	\end{align*}
	First of all, since $(U_\mu)_\mu$ is bounded in $L^\infty((0,T)\times \R)$, we obtain that 
	\[
	\partial_t  \eta_i(U_\mu) + \partial_x q_i(U_\mu) ~\text{is bounded in} ~ W^{-1, \infty}((0,T) \times \R).
	\]
	According to Lemma \ref{lem:dafermos}, it remains to show that $(I_1(\mu))_\mu$ lies in a compact set of $H^{-1}_{\rm loc}(\R_+ \times \R)$, and $(I_2(\mu))_\mu$ lies in a bounded set of the space of measures $\mathcal{M}_{\rm loc}(\R_+ \times \R)$. 
	The second bound directly derives from Lemma \ref{lem:control:etai}.
	About the control of $I_1$, we notice that 
	\begin{align*}
	|\mu \partial_x \eta_i (\rho_{\mu},m_{\mu})|
	& \leq \mu \Big( |\partial_\rho \eta_i(\rho_{\mu},m_{\mu})| \ |\partial_x \rho_{\mu}| 
	+  |\partial_m \eta_i(\rho_{\mu},m_{\mu})| \ |\partial_x m_{\mu}|  \Big) \\
	& \leq C \mu \Big( |\partial_x \rho_{\mu}| +  |\partial_x m_{\mu}|  \Big) \\
	& \leq C \mu \Big(|\partial_x \rho_{\mu}| + |\sqrt{\rho_{\mu}} \ \partial_x u_{\mu}| \Big)
	\end{align*}
	using the bounds \eqref{eq:bound-d-ei-qi}.
	Thanks to Lemma \ref{lem:cvg-dxa} we ensure that $(I_1(\mu))_\mu$ lies in a compact set of $H^{-1}_{\rm loc}(\R_+ \times \R)$, which achieves the proof of Proposition \ref{prop:compact-entropy}.
\end{proof}

\subsubsection{Strong convergence of $U_\mu$, reduction of the Young measure}
Thanks to the bounds derived in Proposition~\ref{prop:bounds-existence}, we infer the weak-* convergence of a subsequence of $(\rho_{\mu},m_{\mu})_\mu$, to which is associated a family of Young measures $(\nu_{(t,x)})_{t,x}$ such that, for any continuous function $h$, 
\[
h(U_{\mu}) \rightharpoonup \bar{h} \quad \text{with} \quad 
\bar{h}(t,x) = \ <\nu_{(t,x)}, h(U)> \ = \int_{\R^2} h(U) \ \d\nu_{(t,x)}(U). 
\]
Passing to the limit in \eqref{eq:sys-v-mu}, we deduce that system \eqref{eq:sing-Euler} is satisfied in the following sense:
\begin{equation}
\partial_t < \nu_{(t,x)},U> + \ \partial_x <\nu_{(t,x)},f_\ep(U)> \ = 0 \quad \text{in} ~  \mathcal{D}'.
\end{equation}
In addition, we notice that the maximal bound satisfied by $\rho_{\mu}$ implies that the weak limit 
\[
\bar{\rho}(t,x) := \ < \nu_{(t,x)}, \rho>
\]
also satisfies
\begin{equation}
0 \leq \bar{\rho}(t,x) \leq A^\ep < 1 \quad \text{a.e.}.
\end{equation}
We recover a weak solution to \eqref{eq:sing-Euler} as $\mu \to 0$ in the sense of Definition \ref{df:sol-sing-Euler} provided that we are able to prove that $\nu_{(t,x)}$ reduces to a Dirac mass, i.e.
\begin{equation*}
\nu_{(t,x)} = \delta_{U(t,x)} \quad \text{a.e.}.
\end{equation*}
In what follows, for the sake of brevity, we shall omit the dependency on $(t,x)$ and we replace $\nu_{(t,x)}$ with $\nu$.
We also denote
\[
\bar{U} = (\brho,\bm) := (<\nu, \rho >\ , <\nu,m> ).
\]

\bigskip
Let us first introduce the pairs
$(\teta_i,\tq_i)=\big(\teta_i(U,\bar{U}),\tq_i(U,\bar{U})\big)$ defined as:

\medskip
\begin{align*}
(\teta_1,\tq_1) & = (\rho - \brho \ , \ m - \bm), \\
(\teta_2,\tq_2) & = \left(m - \bm \ , \ \frac{m^2}{\rho} - \frac{\bm^2}{\brho} + p_\ep(\rho) - p_\ep(\brho) \right),
\end{align*}
\begin{align*}
(\teta_3,\tq_3) & = \Bigg(
\frac{1}{2}\rho \left( \frac{m}{\rho} - \frac{\bm}{\brho} \right)^2 +
\rho \int^\rho \frac{p_\ep(s)}{s^2}  - \brho \int^{\brho} \frac{p_\ep(s)}{s^2} 
- \left( \int^{\brho} \frac{p_\ep(s)}{s^2}  + \frac{p_\ep(\brho)}{\brho} \right) (\rho - \brho)
\ , \\
& \qquad \frac{1}{2}m \left(\frac{m}{\rho} - \frac{\bm}{\brho}\right)^2
+ \left(\frac{m}{\rho} - \frac{\bm}{\brho}\right)(p_\ep(\rho) - p_\ep(\brho)) \\
& \hspace{2cm}+ \frac{m}{\rho}\left(\rho \int_{\brho}^{\rho}\frac{p_\ep(s)}{s^2} - \frac{p_\ep(\brho)}{\brho}(\rho-\brho) \right)
\Bigg), 
\end{align*}
\begin{align*}
(\teta_4,\tq_4) & = \Bigg( 
6m\int_{\brho}^{\rho}\frac{p_\ep(s)}{s^2} + \rho\left(\frac{m}{\rho} - \frac{\bm}{\brho}\right)^2 \left( \frac{m}{\rho} + 2\frac{\bm}{\brho} \right) 
- 6\frac{\bm}{\brho^2}p_\ep(\brho)(\rho -\brho),\\
& \qquad  6\left(\frac{m^2}{\rho} + p_\ep(\rho)\right) \int_{\brho}^{\rho} \frac{p_\ep(s)}{s^2} 
- 6 \int_{\brho}^{\rho} \left(\frac{p_\ep(s)}{s} \right)^2 
+ 3p_\ep(\rho)\left(\frac{m^2}{\rho^2} - \frac{\bm^2}{\brho^2}\right) \\
& \qquad  + m\left(\frac{m}{\rho} - \frac{\bm}{\brho}\right)^2 \left( \frac{m}{\rho} + 2\frac{\bm}{\brho} \right)
- 6 p_\ep(\brho) \frac{\bm}{\brho^2}(m - \bm)
\Bigg),
\end{align*}
Observe that $(\teta_i,\tq_i)$ are related to the pairs $(\eta_i,q_i)$ through the relations 
\[
\begin{cases}
\teta_i(U,\bar{U}) = \eta_i(U) - \eta_i(\bar{U})\\
\tq_i(U,\bar{U}) = q_i(U) - q_i(\bar{U}) 
\end{cases}
\quad \text{for}~ i=1,2,
\]
\[
\begin{cases}
\teta_i(U,\bar{U}) = \eta_i(U) - \eta_i(\bar{U}) - \nabla \eta_i(\bar{U}) . (U-\bar{U}) \\
\tq_i(U,\bar{U}) = q_i(U) - q_i(\bar{U}) - \nabla \eta_i(\bar{U}) . \big(f_\ep(U)- f_\ep(\bar{U})\big)
\end{cases}\quad \text{for}~ i=3,4,
\]
so that we obtain ``relative'' entropy-entropy flux pairs for system \eqref{eq:sing-Euler} (see for instance \cite{brenier2000}, or Dafermos [Section 5.3, \cite{dafermos2000}]).
\begin{lemma}
	For $i=1, \dots, 4$, $(\teta_i,\tq_i)$ is an entropy-entropy flux pair for system \eqref{eq:sing-Euler} satisfying
	\begin{equation}\label{eq:eta-tilde}
	\partial_t \teta_i(U_{\mu}, \bar{U}) + \partial_x \tq_i(U_{\mu}, \bar{U}) 
	\ \subset ~\text{compact set of}~ H^{-1}_{loc}(\mathbb{R}^+ \times \mathbb{R}).
	\end{equation}
	Moreover, we have the identities
	\begin{align}\label{eq:divcurl-id}
	<\nu , \teta_i \tq_j - \teta_j \tq_i> &  = \ <\nu, \teta_i > \ < \nu, \tq_j> 
	- <\nu,\teta_j> <\nu, \tq_i> \quad \forall \ i,j=1, \dots,4.
	\end{align}
\end{lemma}

\begin{proof}
Equations \eqref{eq:eta-tilde} easily derive from Proposition \ref{prop:compact-entropy}. Equipped with \eqref{eq:eta-tilde}, we then apply the Div-Curl lemma \ref{lem:div-curl} to 
\[
G_\mu = (\tq_i(U_\mu), \teta_i(U_\mu)), 
\quad 
H_\mu = (\teta_j(U_\mu), -\tq_j(U_\mu))
\]
observing that
\[
\mathrm{div}_{x,t} G_\mu = \partial_t \teta_i(U_\mu) + \partial_x \tq_i (U_\mu)
\qquad
\mathrm{curl}_{x,t} H_\mu = -\partial_t \teta_j(U_\mu) - \partial_x \tq_j (U_\mu).
\]
\end{proof}

\bigskip
Now, smartly combining the identities \eqref{eq:divcurl-id}, we can show the following lemma.

\begin{lemma}\label{lem:combination}
Consider the pressure
\[
p_\ep(\rho) = \ep \left(\dfrac{\rho}{1-\rho}\right)^\gamma \quad \text{with}~ \gamma \in \ (1, 3].
\]
The following equality holds true
\begin{align} \label{eq:HOHO}
& \ep \dfrac{3-\gamma}{2(\gamma+1)} \dfrac{\brho^{\gamma+1}}{(1-\brho)^\gamma} < \nu, (u -\bu)^4 >
 + \ep^3 \dfrac{\gamma^2(5\gamma +1)}{2(\gamma+1)}  \dfrac{\brho^{3\gamma-5}}{(1-\brho)^{3\gamma+2}} < \nu, (\rho -\brho)^4 >  \nonumber \\
& + \ep^2 \ 6 \gamma  \dfrac{\brho^{2(\gamma-1)}}{(1-\brho)^{2\gamma +1}} \big(< \nu, (u-\bu) (\rho -\brho)> \big)^2  + \ \ER
= 0,
\end{align}
where $\ER$ denotes ``an error term'', whose $L^\infty$ norm is negligible with respect to the norm of the other terms.
\end{lemma}

We postpone to the Appendix the (quite long and technical) proof of Lemma \ref{lem:combination}.

\medskip
\begin{prop}[Reduction of the Young measure]{\label{prop:strong-cvg-mu}}
	Assume that $ \gamma \in (1,3]$.
	The support of $\nu$ is either confined in $\{\brho = 0\}$ or is reduced to the point $(\brho, \bm)$.
\end{prop}

\medskip
	\begin{proof}
		The result of Lemma \ref{lem:combination} at hand, we observe that \eqref{eq:HOHO} implies, if $\gamma \in \ (1,3)$, that
		\begin{align*}
		C_1(\brho) < \nu, (u -\bu)^4 >
		+ C_2(\brho)  < \nu, (\rho -\brho)^4 > 
		+ \ \ER
		\leq 0,
		\end{align*}	
		where the coefficients $C_1(\brho)$ and $C_2(\brho)$ are positive on the set $\{\brho >  0\}$. 
		Therefore, we have
		\begin{align*}
	& 	< \nu_{(t,x)}, (\rho -\brho)^4 > \ = \int_{\R^2}(\rho-\brho)^4 \ \d \nu_{(t,x)}(\rho,m) = 0  \quad \text{a.e. on} ~ \{(t,x)\, | \, \ \brho(t,x) > 0 \}, \\  
	    &< \nu_{(t,x)}, (u -\bu)^4 > \ = \int_{\R^2}\left(\frac{m}{\rho}- \frac{\bm}{\brho} \right)^4 \d \nu_{(t,x)}(\rho,m) = 0
	     \quad \text{a.e. on} ~ \{(t,x)\,| \, \ \brho(t,x) > 0 \},
		\end{align*}
		from which we deduce that
		\begin{equation}
		\nu_{(t,x)} = \delta_{(\brho(t,x),\bm(t,x))}  \quad \text{a.e. on} ~ \{(t,x)\, | \, \ \brho(t,x) > 0 \},
		\end{equation}
		so that the strong convergence of $(\rho_\mu, m_\mu)$ towards $(\brho, \bm)$ is proven.
		
		\medskip
		If $\gamma=3$, then in \eqref{eq:HOHO} the term
		\[\ep \dfrac{3-\gamma}{2(\gamma+1)} \dfrac{\brho^{\gamma+1}}{(1-\brho)^\gamma} < \nu, (u -\bu)^4 >
		\]
		vanishes.
		However, the strong convergence of $u$ can be recovered in that case from the following equality
		\[
		<\nu,(u-\bu)^2>
		= \frac{p'_\ep(\brho)}{\brho^2} <\nu,(\rho-\brho)^2> + \ER,
		\]
		which is proved in the course of the Appendix and given in \eqref{eq:u2}, and where $\ER$ denotes again a negligible term. The proof is over.
	\end{proof}

In the end, the existence of global-in-time weak solutions at $\ep$ fixed to system \eqref{eq:sing-Euler}-\eqref{eq:sing-press-Euler} as in Definition \ref{df:sol-sing-Euler}   is proved by means of a viscous approximation and the compensated compactness method.

\medskip
\begin{rmk}
Note that the previous estimates, and in particular \eqref{eq:HOHO}, strongly depend on $\ep > 0$ and degenerate as $\ep \to 0$.
As a consequence, a similar compactness argument would not work as $\ep \to 0$. 
\end{rmk}


\section{Analysis of the smooth solutions at $\ep$ fixed}{\label{sec:strong-sol-ep}}
The one-dimensional compressible Euler equations in Lagrangian coordinates read as follows
\begin{subnumcases}{\label{eq:psystem-ep}}
\partial_t v_\ep - \partial_x u_\ep = 0, \label{eq:psystem-ep-mass} \\
\partial_t u_\ep + \partial_x p_\ep(v_\ep) = 0, \label{eq:psystem-ep-mom}
\end{subnumcases}
where $v_\ep=1/\rho_\ep$ is the specific volume and the pressure $p_\ep$ (which is in this section a function of $v$) is given by
\begin{equation}\label{df:pressure}
p_\ep(v) = \dfrac{\ep}{(v-1)^\gamma} + \dfrac{\kappa}{v^{\tgamma}}=:p_{\ep,1}(v)+p_{2}(v),
\end{equation}
where $\kappa >0$, $\gamma > 1$, $\tgamma \in (1,3)$ and $\ep \leq \ep_0$ is a positive small parameter.\\
The characteristic speeds of system \eqref{eq:psystem-ep} are
\begin{equation}\label{def:eigenvalues-2}
\pm c_\ep = \pm \sqrt{-p'_\ep(v_\ep)} = \pm \sqrt{\dfrac{\ep \gamma}{(v_\ep -1)^{\gamma+1}} + \dfrac{\kappa \tgamma}{v_\ep^{\tgamma+1}}}
\end{equation}
and, introducing the quantity
\begin{equation}
\theta_\ep(v) := \int_{v}^\infty c_\ep(\tau) \d\tau ,
\end{equation}
the Riemann invariants of system \eqref{eq:psystem-ep} read
\begin{equation}\label{df:Riemann-inv}
w_\ep = u_\ep + \theta_\ep(v_\ep), \qquad 
z_\ep = u_\ep - \theta_\ep(v_\ep).
\end{equation}

We start by introducing the set of assumptions on the initial data, which are needed throughout this section. 
\subsection{Initial data setup}\label{sub-initial-data}
We gather in this subsection all the initial assumptions that are required in both Sections \ref{sec:strong-sol-ep} and \ref{sec:strong-sol-lim}.

\begin{assumption}
\label{ass-I-setup}
For any $\ep> 0$, $(v^0_\ep,u^0_\ep)$ are $\mathcal{C}^1$ functions and that there exist $M_1,M_2 > 0$, independent of $\ep$, such that
\begin{equation}\label{hyp:init-data-sec}
(v^0_\ep-1)^{\gamma-1} \geq M_1^{-1} \ep , \qquad \|(v^0_\ep, u^0_\ep)\|_{L^\infty(\R)}+ \|(\partial_x v^0_\ep, \partial_x u^0_\ep)\|_{L^\infty(\R)} \leq M_2.
\end{equation}
\end{assumption}

\medskip
As mentioned in the introduction, the first inequality allows us to provide a control of the Riemann invariants $w_\ep$, $z_\ep$.


We assume further conditions on the Riemann invariants at the initial time:

\begin{assumption}
\label{ass-III-setup}
There exist two constant $Y^0, Q^0$ independent of $\ep$ such that
\begin{equation}\label{hyp:dxw}
\sqrt{c_\ep^0} \partial_x w_\ep^0 \leq Y^0 , \qquad \sqrt{c_\ep^0} \partial_x z_\ep^0 \leq Q^0.
\end{equation}
\end{assumption}


\medskip
Finally, anticipating on Section~\ref{sec:strong-sol-lim}, we make the following two assumptions.
\begin{assumption}
\label{ass-IV-setup}
Assume initially that
\begin{equation}\label{hyp:init-data-3}
\left(\dfrac{\ep}{(v_\ep^0 -1)^{\gamma+1}}\right)^{\frac{1}{4}} \ \Big( [\partial_x w_\ep^0]_- + [\partial_x z_\ep^0]_-\Big)
=
\begin{cases}
\mathcal{O}\big(\ep^{\frac{1}{2(\gamma-1)}}\big) \quad & \text{if} \quad \gamma \in (1,3), \\
\mathcal{O}\big(\ep^{\frac{1}{4}}\big) \quad & \text{if} \quad \gamma = 3, \\
\mathcal{O}\big(\ep^{\frac{1}{\gamma+1}}\big) \quad & \text{if} \quad \gamma > 3.
\end{cases}
\end{equation}
where $[f]_- = \max(-f,0)$.
\end{assumption}

\begin{assumption}
\label{ass-II-setup}
There exists a constant value $v_\pm > 1$ independent of $\ep$ so that
\begin{equation}\label{hyp-init-data-B}
\lim_{x_\rightarrow \pm \infty} v^0_\ep(x) = v_\pm.
\end{equation}
In addition, there exist $\ell^*, \underline{v} > 0$, both independent of $\ep$, such that
\begin{equation}\label{hyp:init-data-4}
\dfrac{1}{2\ell} \int_{-\ell}^{\ell} v^0_\ep(x) \, \d x \geq \underline{v} > 1 \qquad \text{for all} \qquad \ell \geq \ell^* .
\end{equation}
\end{assumption}

As it will be clarified in Subsection \ref{sub-existence}, Assumption \ref{ass-IV-setup} is designed so that the maximal time of existence of smooth solutions to \eqref{eq:psystem-ep} can be $\ep$-uniformly bounded from below. 
Assumption~\ref{ass-II-setup} is a technical hypothesis used for providing an $L^1$ local (in time and space) control of the pressure $p_\ep(v_\ep)$ in Subsection \ref{subsect-pressure} (see analogous conditions in~\cite{peza2015} for instance).

\medskip
\begin{rmk}[Consequences of Assumptions~\ref{ass-IV-setup} on $\partial_x u^0_\ep$]
As a direct consequence of \eqref{hyp:init-data-3} and the relation $\partial_x w^0_\ep + \partial_x z^0_\ep = \partial_x u^0_\ep$, we have
\begin{equation} \label{eq:bound-u0-ep}
\left(\dfrac{\ep}{(v_\ep^0 -1)^{\gamma+1}}\right)^{\frac{1}{4}} \  [\partial_x u_\ep^0]_- 
\lesssim
\begin{cases}
~ \ep^{\frac{1}{2(\gamma-1)}} \quad & \text{if} \quad \gamma \in (1,3), \\
~ \ep^{\frac{1}{4}} \quad & \text{if} \quad \gamma = 3, \\
~ \ep^{\frac{1}{\gamma+1}} \quad & \text{if} \quad \gamma > 3.
\end{cases}
\end{equation}
Hence, in the regions close to the congestion constraint, i.e. where $v_\ep^0(x)- 1 = \mathcal{O}(\ep^\alpha)$ with $\alpha > 0$, 
Assumption~\ref{ass-IV-setup} implies
\begin{equation*}
[\partial_x u_\ep^0]_- 
\lesssim
\begin{cases}
~ \ep^{\frac{3-\gamma}{4(\gamma-1)} + \frac{\alpha}{4}(\gamma+1) } \quad & \text{if} \quad \gamma \in (1,3), \\
~ \ep^{\frac{\alpha}{4}(\gamma+1)} \quad & \text{if} \quad \gamma = 3, \\
~ \ep^{\frac{3-\gamma}{4(\gamma+1)} + \frac{\alpha}{4}(\gamma+1)} \quad & \text{if} \quad \gamma > 3,
\end{cases}
\end{equation*}
forcing therefore $[\partial_x u_\ep^0]_-$ to be small in terms of $\ep$ (except in the case $\gamma > 3$).
The previous bounds~\eqref{eq:bound-u0-ep} for $\gamma \leq 3$ lead then to the following condition on the limit initial datum 
\[
\partial_x u^0 \geq 0 \quad \text{a.e. on} \quad \{v^0=1\}.
\]
We recover here a kind of ``compatibility'' condition on the initial datum: one can only ``dilate'' the medium (in the mechanical sense $\Div u \geq 0$) in the congested/saturated regions. 
\end{rmk}

\medskip


\subsection{Invariant regions: lower and upper bounds}{\label{sec:inv-smooth}}

Aiming at obtaining uniform bounds, the next step is to rearrange system \eqref{eq:psystem-ep} in terms of the Riemann invariants \eqref{df:Riemann-inv}, so that 
\begin{subnumcases}{\label{eq:riemann-inv-transp}}
\partial_t w_\ep + c_\ep \partial_x w_\ep = 0, \\
\partial_t z_\ep - c_\ep \partial_x z_\ep = 0.
\end{subnumcases}
It is now an easy task to get an a priori lower bound for the specific volume $v_\ep$ and an upper bound for the velocity $u_\ep$ as follows.

\begin{lemma}{\label{lem:bounds-1}}
Under Assumption \ref{ass-I-setup}, there exists
two positive constants $C_1, C_2>0$ independent of $\ep$, such that
\begin{equation}\label{eq:lower-bound-vep}
v_\ep \geq 1 + C_1\, \ep^{\frac{1}{\gamma-1}}.
\end{equation}
and 
\begin{equation}
\label{eq:u-Linfty}
\|u_\ep\|_{L^\infty_{t,x}} \leq C_2.
\end{equation}
\end{lemma}

\bigskip
\begin{proof}
From the definition of the Riemann invariants \eqref{df:Riemann-inv},
\begin{align*}
w_\ep^0 = u_\ep^0+\theta_\ep(v_\ep^0),
\end{align*}
and from Assumption \ref{ass-I-setup},
\begin{align*}
    u_\ep^0 \le M_2, \qquad \theta_\ep(v_\ep^0) \le C(M_1).
\end{align*}
Hence $\|w_\ep^0\|_{L^\infty} \leq M$ where $M$ is independent of $\ep$. 
Observing that $w_\ep$ and $z_\ep$ are constant along the characteristics, it is now classical to show that the domain is invariant
\begin{equation}\label{eq:invariant-region}
\Sigma := \{(v_\ep,u_\ep), ~ w_\ep \leq M, ~ z_\ep \geq - M\}.
\end{equation}
This implies that
\begin{equation}\label{eq:bound-eta}
\theta_\ep(v_\ep) \leq 2M,
\end{equation}
which directly yields the lower bound on $v_\ep$
in \eqref{eq:lower-bound-vep}.
One also has the control of the velocity
\begin{equation}\label{eq:bound-u}
\|u_\ep\|_{L^\infty} \le \|w_\ep\|_{L^\infty} + \|\theta_\ep(v_\ep)\|_{L^\infty} \le 3M, 
\end{equation}
which concludes the proof.
\end{proof}


\subsection{A uniform upper bound on the specific volume}
Now we introduce the following change of variables, due to \cite{chen2017},
\begin{equation}\label{def:Riccati-variables}
y_\ep := \sqrt{c_\ep} \partial_x w_\ep, \qquad q_\ep = := \sqrt{c_\ep} \partial_x z_\ep.
\end{equation}
In terms of the new variables $(y_\ep,q_\ep)$, system \eqref{eq:psystem-ep} read
\begin{subnumcases}{\label{eq:riccati-system}}
\partial_t y_\ep + c_\ep \partial_x y_\ep
 =a_\ep \, y_\ep^2, \label{eq:riccati-yep} \\
\partial_t q_\ep - c_\ep \partial_x q_\ep = - a_\ep \, q_\ep^2,\label{eq:riccati-qep} 
\end{subnumcases}
where 
\begin{align}\label{def:a-ep}
a_\ep = a_\ep(v_\ep)= - \dfrac{c'_\ep(v_\ep)}{2\sqrt{c_\ep(v_\ep)}c_\ep(v_\ep)} = \dfrac{p''_\ep(v_\ep)}{2(-p'_\ep(v_\ep))^{5/4}}.
\end{align}


We provide an $\ep$-uniform upper bound on the specific volume $v_\ep$.
\begin{lemma}[Upper bound on $v_\ep$]\label{lem:upper-bound-vep}
	Let $(v_\ep,u_\ep)$ belonging to  $C^1_{t,x}=\mathcal{C}^1([0,T]\times\R)$ be a solution to system \eqref{eq:psystem-ep} on the time interval $[0,T]$, with initial data satisfying Assumption \ref{ass-I-setup} and Assumption \ref{ass-III-setup}.
	Then, there exist $K= K(\kappa),\bar{Y}, \bar{Q} > 0$, independent of $\ep$, such that
	\begin{equation}\label{eq:bound-v}
	v_\ep(t,x) \leq \Big(v^0_\ep(x)^{\frac{3-\tgamma}{4}} + K(\bar{Y} + \bar{Q}) t \Big)^{\frac{4}{3-\tgamma}}
	\quad \forall \ t \in [0,T].
	\end{equation}
\end{lemma}

\bigskip
\begin{proof}
	By comparison principle for ODEs, we ensure thanks to Assumption \ref{ass-I-setup}-\ref{ass-III-setup} that
	\begin{equation}\label{eq:boundYQ}
	y_\ep(t,x) \leq \bar{Y} = \max \big\{0, Y^0 \big\}, \qquad q_\ep(t,x) \leq \bar{Q} = \max \big\{0, Q^0 \big\}.
	\end{equation}
	Now since
	\begin{align*}
	y_\ep + q_\ep =2 \sqrt{c_\ep} \partial_x u_\ep = 2 \sqrt{c_\ep} \partial_t v_\ep, 
	\end{align*}
	so that, using $c_\ep > \sqrt{\kappa \tgamma} \ (v_\ep)^{-(\tgamma+1)/2}$,
	we deduce that
	\[
	(\kappa\tgamma)^{1/4} \ v_\ep^{-\frac{\tgamma+1}{4}} \partial_t v_\ep
	\leq \dfrac{1}{2}(\bar{Y} + \bar{Q}).
	\]
	Integrating in time, we obtain the desired (time-dependent) upper bound in \eqref{eq:bound-v} using that $\tgamma \in (1,3)$ (see Remark \ref{rmk-tgamma}), and the proof is concluded.
\end{proof}


\subsection{Existence of regular solutions: non-compressive and compressive case}\label{sub-existence}
In this section, we provide an analysis of the smooth solutions to system \eqref{eq:psystem-ep}. Two different situations are identified. 
We rely on Definition \ref{def:noncompressive-case} presented at the beginning.

\begin{thm}\label{thm:global-ex-ep} Under Assumption \ref{ass-I-setup}-\ref{ass-III-setup}, we obtain the following dichotomy result.
\begin{itemize}
	\item If the initial datum is everywhere rarefactive in the sense of Definition \ref{def:noncompressive-case}, then
	 there exists a unique global-in-time $\mathcal{C}^1_{t,x}$ solution $(v_\ep,u_\ep)$, whose $\mathcal{C}^1_{t,x}$-norm is independent of $\ep$.
	\item Otherwise, there exists a unique local $\mathcal{C}^1_{t,x}$ solution $(v_\ep,u_\ep)$ which breaks down in finite time $T^*= T^*(\ep) < + \infty$.   
\end{itemize}
\end{thm}

\bigskip

The proof of this theorem relies on the following Lemma.

\medskip
\begin{lemma}\label{lem:a-ep}
Let $\kappa > 0$ be fixed, $\ep \leq \ep_0$ small enough, and consider for $v_\ep > 1$:
\[
a_\ep(v_\ep) = \dfrac{p''_\ep(v_\ep)}{2(-p'_\ep(v_\ep))^{5/4}}
\quad \text{where} \quad p_\ep(v_\ep) = \dfrac{\ep}{(v_\ep-1)^\gamma} + \dfrac{\kappa}{v_\ep^{\tgamma}}=p_{\ep, 1}(v_\ep)+p_{2}(v_\ep), 
\] 
with $\gamma > 1, \ \tgamma \in (1,3).$ We distinguish three main cases:
\begin{enumerate}
	\item Case where $v_\ep-1 = \mathcal{O}(\ep^\alpha)$ with $\dfrac{1}{\gamma+1} \leq \alpha \leq \dfrac{1}{\gamma-1}$.
	There exist two positive constants $K_1, K_2$, independent of $\ep$, such that
	\begin{align}
	K_1 \ep^{-\frac{1}{\gamma+1}} \leq a_\ep(v_\ep) \leq K_2\ep^{-\frac{1}{2(\gamma-1)}} \quad & \text{if} \quad \gamma \in (1,3), \label{eq:a-up-1}\\
	K_1 \ep^{-\frac{1}{4}} \leq a_\ep(v_\ep) \leq K_2\ep^{-\frac{1}{4}} \quad & \text{if} \quad \gamma = 3 ,\label{eq:a-up-2} \\
	K_1 \ep^{-\frac{1}{2(\gamma-1)}} \leq a_\ep(v_\ep) \leq K_2\ep^{-\frac{1}{\gamma+1}} \quad & \text{if} \quad \gamma > 3.\label{eq:a-up-3}
	\end{align}
	\item Case where $v-1 = \mathcal{O}(\ep^\alpha)$ with $\dfrac{1}{\gamma+2} < \alpha < \dfrac{1}{\gamma+1}$. There exist two positive constants $K_1, K_2$, independent of $\ep$, such that
	\begin{align}
	K_1 < K_1 \ep^{1-\alpha(\gamma+2)} \leq a_\ep(v_\ep) \leq K_2\ep^{1-\alpha(\gamma+2)} < K_2 \ep^{-\frac{1}{\gamma+1}};
	\end{align}
	\item Case where $\ep^{\frac{1}{\gamma+2}}\lesssim  v_\ep-1 \leq v_{max}-1$. There exist two positive constants $K_1, K_2$, independent of $\ep$, such that
	\[
	K_1 v_{max}^{\frac{\tgamma-3}{4}} \leq a_\ep(v_\ep) \leq K_2.
	\]
\end{enumerate}
\end{lemma}

\bigskip
\begin{proof} The bounds on $a_\ep$ are directly derived from the expression of the pressure, which simplifies according to the considered regime, Case 1, 2 or 3.
In each case, one of the two components of the pressure law is indeed negligible.
\begin{enumerate}
	\item In the first regime, the singular component $p_{\ep, 1}(v_\ep)$ is dominant both in $p''_\ep$, $p'_\ep$:
	\[
	p''_\ep(v_\ep) \underset{\ep \to 0}{\sim} \gamma (\gamma+ 1)\dfrac{\ep}{(v_\ep-1)^{\gamma+2}}, \qquad 
	p'_\ep(v_\ep) \underset{\ep \to 0}{\sim} -\gamma \dfrac{\ep}{(v_\ep-1)^{\gamma+1}};
	\]
	Thus, 
	\[
	a_\ep(v_\ep) \underset{\ep \to 0}{\sim} \ep^{-\frac{1}{4}(1+\alpha (3-\gamma))},
	\]
	which directly provides the bounds of Case 1 using that $\gamma \in (1, 3)$.
	\item In the intermediate regime, since for $\dfrac{1}{\gamma+2} < \alpha < \dfrac{1}{\gamma+1}$
	\[
	p_\ep'(v_\ep)=-\gamma \ep^{1-\alpha(\gamma+1)}-\kappa \tgamma v_\ep^{-(\tgamma+1)} \underset{\ep \to 0}{\sim} -\tgamma \dfrac{\kappa}{v_\ep^{\tgamma+1}}=p'_{ 2}(v_\ep),
	\]

	then the singular component $p_{\ep, 1}(v_\ep)$ is dominant only in $p''_\ep$,
	\[
	p''_\ep(v_\ep) \underset{\ep \to 0}{\sim} \gamma (\gamma+ 1)\dfrac{\ep}{(v_\ep-1)^{\gamma+2}}.
	\]
	\item In the last regime, since $v_\ep$ is ``far'' from $1$, the component $p_{\ep, 1}(v_\ep)$ in negligible in both $p''_\ep$ and $p'_\ep$ as $\ep \to 0$. 
	The bounds on $a_\ep$ are directly derived from the upper and lower bounds on $v_\ep$ in $p_{2}(v_\ep)$, using the fact that $\tgamma \in (1,3)$.
\end{enumerate}
For sake of brevity, we omit further details.
\end{proof}

\bigskip
\begin{proof}[Proof of Theorem \ref{thm:global-ex-ep}]
Under Assumption \ref{ass-I-setup} on the initial data, one can prove by classical arguments (see for instance [Section 7.8, \cite{dafermos2000}] the local existence of a unique $\mathcal{C}^1$ solution $(v_\ep,u_\ep)$.

\medskip

\textit{The rarefactive case}.
We want to extend the previous local solution $(v_\ep,u_\ep)$ by a continuity argument. 
For that purpose, we need to show a priori controls of the $L^\infty$ and Lipschitz norms of $(v_\ep,u_\ep)$.
Let us recall the result of Lemma \ref{lem:bounds-1}: from the bounds on $w_\ep$ and $z_\ep$ (see \eqref{eq:invariant-region}), we infer a control in $L^\infty$ (uniform in $\ep$) on $u_\ep$ as well as a lower bound \eqref{eq:lower-bound-vep} on $v_\ep$.
Hence, it remains to show the control on $(\partial_x v_\ep, \partial_x u_\ep)$.\\
From hypothesis \eqref{hyp:init-compr}, we have $(y_\ep)_{|t=0} \geq 0$, $(q_\ep)_{|t=0} \geq 0$.
We ensure then $y_\ep(t,x) \geq 0$, $q_\ep(t,x) \geq 0$ for all times $t \geq 0$ and, recalling \eqref{eq:boundYQ}, we deduce that
\[
\|y_\ep(t, \cdot)\|_{L^\infty} \leq \bar{Y}, \qquad \|q_\ep(t, \cdot)\|_{L^\infty} \leq \bar{Q}.
\]
Thanks to Lemma~\ref{lem:upper-bound-vep}, we have the following upper bound on $v_\ep$ (recall that $\tgamma \in (1,3)$):
\[
v_\ep(t,x) \leq \Big(v^0_\ep(x)^{\frac{3-\tgamma}{4}} + K(\bar{Y} + \bar{Q}) t \Big)^{\frac{4}{3-\tgamma}}
\]
and therefore
\[
c_\ep(t,x) 
\geq \sqrt{\dfrac{\kappa \tgamma}{v_\ep(t,x)^{\tgamma+1}}} 
\geq \underline{c}(T) \qquad \forall \ t \leq T.
\]
This lower bound on $c_\ep$ allows us to control $\partial_x w_\ep$:
\[
\|\partial_x w_\ep(t,\cdot)\|_{L^\infty} \leq \left\|\frac{y_\ep(t,\cdot)}{\sqrt{c_\ep(t, \cdot)} }\right\|_{L^\infty} \leq \overline{K}(t)\bar{Y} \leq \overline{K}(t),
\]
where, hereafter, $\overline{K}$ denotes a generic function of time which is bounded uniformly w.r.t $\ep$ for any finite time $t$.
Similarly,
\[
\|\partial_x z_\ep(t,\cdot)\|_{L^\infty} \leq  \left\|\frac{q_\ep(t,\cdot)}{\sqrt{c_\ep(t, \cdot)} }\right\|_{L^\infty} \leq \overline{K}(t).
\]
Combining these two bounds then leads to the control of $\partial_x u_\ep$:
\begin{equation}\label{eq:bound-dx-u}
\|\partial_x u_\ep(t,\cdot)\|_{L^\infty} = \|\partial_x (w_\ep + z_\ep)(t,\cdot)\|_{L^\infty} \leq \overline{K}(t).
\end{equation}
Finally, since
\begin{equation*}
\theta'_\ep(v_\ep) \partial_x v_\ep = \partial_x w_\ep - \partial_x u_\ep 
\quad \text{with} \quad 
\theta'_\ep(v_\ep) = - c_\ep(v_\ep)
\end{equation*}
we also control $\partial_x v_\ep$
\begin{equation}\label{eq:bound-dx-v}
\|\partial_x v_\ep(t,\cdot)\|_{L^\infty} \leq \overline{K}(t).
\end{equation}
Now, let us assume that the local solution $(v_\ep,u_\ep)$ admits a finite maximal existence time $T^*<+\infty$.
Since $T^*$ is finite, then $\underline{c}(T^*) > 0$. As a consequence $\overline{K}(T^*) < +\infty$ and the spatial $\mathcal{C}^1_x$-norm of $(v_\ep(T^*, \cdot),u_\ep(T^*, \cdot))$ is controlled, Assumption \ref{ass-I-setup}-\ref{ass-III-setup} are satisfied at time $T^*$.
Applying once again the local existence result starting at time $T^*$, we deduce that there exists $t^*>0$ such that the solution $(v_\ep,u_\ep)$ can be extended on the time interval $[0,T^*+ t^*)$.   
This is in contradiction with the fact that $T^*$ is the maximal time of existence. 
Hence, the solution $(v_\ep,u_\ep)$ exists globally in time and the previous estimates show that its $\mathcal{C}^1_{t,x}$-norm is independent of $\ep$.

\bigskip
\noindent
\textit{The compressive case.}
This is the case where \eqref{hyp:init-compr} is not satisfied, i.e. there exists a point $x^* \in \R$ such that
\[
y_\ep(0,x^*) = \sqrt{c_\ep^0(x^*)} \partial_x w^0_\ep(x^*) < 0 \quad \text{or} \quad q_\ep(0,x^*)= \sqrt{c_\ep^0(x^*)} \partial_x z^0_\ep(x^*) < 0.
\]
Let us consider the case where $y_\ep(0,x^*) < 0$ ($q_\ep(0,x^*) < 0$ is analogous). As a consequence of the Riccati equation \eqref{eq:riccati-yep} one has, as long as the solution exists,
\begin{equation}\label{eq:yep-sol}
\dfrac{1}{y_\ep(t,x^+_\ep(t))} = \dfrac{1}{y_\ep(0,x^*)} + \int_0^t a_\ep(\tau, x^+_\ep(\tau)) \, \d\tau
\end{equation}
where $x^+_\ep$ is the forward characteristic emanating from $x^*$, i.e. 
\[
\dfrac{\d x^+_\ep}{\d t}(t) = c_\ep(t,x^+_\ep), \qquad x^+_\ep(0) = x^*.
\]
This way, the appearance of a singularity in $y_\ep$ essentially depends on the function $a_\ep$, whose asymptotic is detailed in Lemma~\ref{lem:a-ep}.
Since by Lemma \ref{lem:upper-bound-vep},  
\[
v_\ep(t,x) \leq \Big(\|v^0_\ep\|_{L^\infty}^{\frac{3-\tgamma}{4}} + K(\bar{Y} + \bar{Q}) t \Big)^{\frac{4}{3-\tgamma}} =: v_{max}(t),
\]
then we observe that, in all cases,
\begin{equation}
\int_0^{T} a_\ep(t,x_\ep^+(t))\ \d t \longrightarrow + \infty \quad \text{as} \quad T \to +\infty.
\end{equation}
As a consequence, if there exists a point $x^*\in \R$ such that $y_\ep(0,x^*) < 0$, then there exists a finite time $T^*_\ep$ such that
\begin{align}\label{def:maximal-time}
\int_0^{T_\ep^*} a_\ep(t,x_\ep^+(t))\ \d t = - \dfrac{1}{y_\ep(0,x^*)} > 0,
\end{align}
whence $y_\ep(t,x^+_\ep(t)) \to -\infty$ as $t \to T_\ep^*$.\\
Using the same arguments, if there exists $x^*\in \R$ such that $q_\ep(0,x^*) < 0$, one can show the existence of a finite time $T^*_\ep$ such that
\begin{align*}
\int_0^{T_\ep^*} a_\ep(t,x_\ep^-(t))\ \d t = - \dfrac{1}{q_\ep(0,x^*)}
\end{align*}
where $x_\ep^-$ is the backward characteristic emanating from $x^*$, i.e s.t. $(x_\ep^-)'(t) = -c_\ep$ and $x_\ep^-(0)= x^* $.
This achieves the proof of the second part of Theorem \ref{thm:global-ex-ep}.
\end{proof}

\bigskip
\begin{rmk}
\label{rmk-explosion}
From the definitions of $y_\ep, q_\ep$ in \eqref{def:Riccati-variables},
\[
y_\ep = \sqrt{c_\ep} \partial_x w_\ep = \sqrt{c_\ep} (\partial_x u_\ep + \partial_x \theta_\ep(v_\ep)), \quad 
q_\ep = \sqrt{c_\ep} (\partial_x u_\ep - \partial_x \theta_\ep(v_\ep)),
\]
we observe that the explosion of $y_\ep$ (or $q_\ep$) may \emph{a priori} correspond to either the explosion of $c_\ep$ or the explosion of $\partial_x w_\ep$. 
Provided that Assumption \ref{ass-I-setup}-\ref{ass-III-setup} are satisfied, Lemma \ref{lem:upper-bound-vep} with   \eqref{eq:lower-bound-vep} yields that
\begin{equation}\label{eq:upper-bound-cep}
c_\ep(v_\ep) \lesssim \ep^{-\frac{1}{\gamma-1}}.
\end{equation}
Therefore, at $\ep>0$ fixed, when a finite-time singularity occurs in the system, it impacts the spatial derivatives $(\partial_x v_\ep,\partial_x u_\ep)$.
\end{rmk}

\section{Singular limit in the smooth setting}{\label{sec:strong-sol-lim}}
In this section, we aim at justifying the limit in the vanishing $\ep$ parameter of the singular p-system~\eqref{eq:psystem-ep}. 
To this end, since we know from Theorem \ref{thm:global-ex-ep}
that the maximal time of existence of the smooth solution $(u_\ep, v_\ep)$ can be finite at $\ep$ fixed, we need to make sure that it is $\ep$-uniformly bounded from below, and does not shrink to zero as $\ep$ vanishes. This is indeed proved in Proposition \ref{prop:inf-T*}. Later, we employ Assumption \ref{ass-II-setup} (together with Assumption \ref{ass-I-setup}) to obtain a uniform control of the singular pressure. 
Putting all these results together, we finally pass to the limit in Subsection~\ref{sec:epto0}.

\subsection{Lower bound on the maximal existence time}

\begin{prop}\label{prop:inf-T*}
Let Assumptions \ref{ass-I-setup}-\ref{ass-III-setup}-\ref{ass-IV-setup} hold.
Then there exists $T > 0$, independent of $\ep$, such that, for any $\ep < \ep_0$, the smooth solution $(v_\ep,u_\ep)$ provided by Theorem \ref{thm:global-ex-ep} exists on the whole interval $[0,T]$.
\end{prop}

\begin{proof}
In the case where the initial datum  is \emph{everywhere rarefactive}, we know from Theorem \ref{thm:global-ex-ep} that the smooth solution to \eqref{eq:psystem-ep} exists for all times. 
Then we need to handle the \emph{compressive case}. More precisely,
we need to show that we can bound from below, uniformly in $\ep$, the maximal time of existence $T^*_\ep$ when there is some compression in the initial data. 
Let us assume for instance that $y_\ep(0,x^*) < 0$. Then by \eqref{eq:yep-sol} we have at the explosion time $T^*_\ep$ that
\begin{equation}\label{eq:T*ep}
\int_0^{T^*_\ep} a_\ep(t,x_\ep^+(t)) \d t = - \dfrac{1}{y_\ep(0,x^*)}.
\end{equation}
To derive a lower bound on $T^*_\ep$, we need an estimate of $a_\ep(v_\ep)$, which is in fact provided by Lemma \ref{lem:a-ep} and depends on the distance between $v_\ep$ and $1$.
In this regard, Lemma \ref{lem:a-ep} distinguishes three main cases, which we analyse here.

\medskip
\begin{itemize}
\item \textit{Case 3 of Lemma \ref{lem:a-ep}.} We begin with the case where
$v_\ep^0(x^*) \gtrsim 1 + \ep^{\frac{1}{\gamma+2}}$, namely the initial specific volume evaluated at the point $x^*$ is ``far'' from $1$. 
Then, Lemma~\ref{lem:a-ep} ensures the existence a constant $K_2 > 0$, independent of $\ep$ such that
\[
0 < a_\ep(0, x^*) \leq K_2.
\]
On the hand, thanks to Assumption \ref{ass-I-setup},
we also have that
\[
y_\ep(0,x^*) 
= \sqrt{c_\ep^0(x^*)} \ \Big(\partial_x u_\ep^0(x^*) + \theta'_\ep(v_\ep^0(x^*)) \partial_x v_\ep^0(x^*)\Big)
\geq - K_3
\]
for some constant $K_3 > 0$ which is independent of $\ep$.
We have then two possibilities. The first option is that $v_\ep$ remains ``far'' from the congestion constraint (i.e. in Case 3) until the singularity occurs at $T^*_\ep$. In this case, 
$a_\ep$ remains bounded from above uniformly with respect to $\ep>0$ on the whole time interval $[0,T^*_\ep)$, i.e.
\begin{align*}
    0 < a_\ep(v_\ep(t,x)) \le K_4, \quad t \in [0, T_\ep^*),
\end{align*}
with $K_4>0$ independent of $\ep$.
Hence, using again \eqref{eq:T*ep}, we infer the desired $\ep$-uniform lower bound on $T_\ep^*$
\begin{equation}\label{eq:low-Tep-1}
T_\ep^* 
\geq  \dfrac{1}{ \displaystyle \sup_{t \in [0,T^*_\ep)} a_\ep(t,x_\ep^+(t))} \int_0^{T^*_\ep} a_\ep(t,x_\ep^+(t)) \d t 
\geq - \dfrac{1}{K_4} \dfrac{1}{y_\ep(0,x^*)}
\geq  \dfrac{1}{K_4} \dfrac{1}{K_3}.
\end{equation}
The alternative scenario is that, at some time $t^* < T^*_\ep$, $v_\ep$ gets closer to the congestion threshold, passing through the intermediate regime of Case 2. 
This would imply that $v_\ep(t^*,x_\ep^+(t^*)) < 1 + \ep^{\alpha}$ with $\alpha > \frac{1}{\gamma+2}$.
In this case, by continuity of the solution $(v_\ep,u_\ep)$, we can find a positive time $\bar{t} < t^*$ such that
\[
0 < a_\ep(t,x^+_\ep(t)) \leq 2 K_2, \quad \forall \ t \in [0,\bar{t}].
\]
Replacing $T^*_\ep$ by $\bar{t}$ and $K_4$ by $2K_2$ in \eqref{eq:low-Tep-1}, we deduce that
\begin{align*}
    \bar{t} \ge \dfrac{1}{2K_2 K_3},
\end{align*}
namely $\bar{t}$ is bounded from below uniformly with respect to $\ep$, and $T^*_\ep > t^* > \bar{t}$ as well.

\medskip
\item 
\textit{Case 1 and 2 of Lemma \ref{lem:a-ep}.}
Let us now deal with the worst scenario: the case where $v^0_\ep(x^*)$ is close to the congestion constraint, namely
\[
v_\ep^0(x^*) - 1 = \mathcal{O}\big(\ep^{\alpha}\big), \quad \dfrac{1}{\gamma+2} \le \alpha \le \dfrac{1}{\gamma-1}.
\] 

If at some time $\bar{t} < T^*_\ep$, $v_\ep(\bar{t}, x^+_\ep(\bar{t}))$ escapes from this domain, i.e. the specific volume gets away from the congestion constraint, then we are back to the previous case and we bound from below $T^*_\ep$.
So, we assume that 
\[
v_\ep(t, x^+_\ep(t)) - 1 = \mathcal{O}\big(\ep^{\frac{1}{\gamma+1}}\big), \quad \dfrac{1}{\gamma+1} \le \alpha \le \dfrac{1}{\gamma-1}, \quad \forall \ t \in [0, T^*_\ep).
\] 
Let us recall that from Lemma \ref{lem:a-ep}, we have
\begin{align*}
a_\ep(t, x^+_\ep(t)) \leq K_2\ep^{-\frac{1}{2(\gamma-1)}} \quad & \text{if} \quad \gamma \in (1,3), \\
a_\ep(t, x^+_\ep(t)) \leq K_2\ep^{-\frac{1}{4}} \quad & \text{if} \quad \gamma = 3 , \\
a_\ep(t, x^+_\ep(t)) \leq K_2\ep^{-\frac{1}{\gamma+1}} \quad & \text{if} \quad \gamma > 3;
\end{align*}
and thus 
\begin{equation*}
T_\ep^* 
\geq - \dfrac{1}{ \displaystyle \sup_{t \in [0,T^*_\ep[} a_\ep(t,x_\ep^+(t))} \dfrac{1}{y_\ep(0,x^*)}
\geq \begin{cases}
 - \dfrac{\ep^{\frac{1}{2(\gamma-1)}}}{K_2 \ y_\ep(0,x^*)} \quad & \text{if} \quad \gamma \in (1,3), \\
 - \dfrac{\ep^{\frac{1}{4}}}{K_2 \ y_\ep(0,x^*)} \quad & \text{if} \quad \gamma = 3, \\
 - \dfrac{\ep^{\frac{1}{\gamma+1}}}{K_2 \ y_\ep(0,x^*)} \quad & \text{if} \quad \gamma > 3.
\end{cases} 
\end{equation*}
Thanks to Assumption \ref{ass-III-setup}, we guarantee in the three cases that $y_\ep(0,x^*)$ will be small enough to compensate the blow up of $a_\ep$ as $\ep \to 0$, namely
\[
|y_\ep(0,x^*)|
= \sqrt{c_\ep^0(x^*)} |\partial_x w_\ep^0(x^*)|
=
\begin{cases}
\mathcal{O}\big(\ep^{\frac{1}{2(\gamma-1)}}\big) \quad & \text{if} \quad \gamma \in (1,3), \\
\mathcal{O}\big(\ep^{\frac{1}{4}}\big) \quad & \text{if} \quad \gamma = 3, \\
\mathcal{O}\big(\ep^{\frac{1}{\gamma+1}}\big) \quad & \text{if} \quad \gamma > 3.
\end{cases}
\]
Hence, we obtain a lower bound on $T^*_\ep$ which is uniform with respect to $\ep$. Notice that this is the point where Assumption \ref{ass-IV-setup} plays its key role in providing an $\ep$-uniform lower bound on the maximal existence time.

\end{itemize}
\end{proof}

\bigskip
From now on, we shall consider the time interval $[0,T]$ on which the whole sequence of solutions $(v_\ep,u_\ep)_\ep$ exists, $T$ being independent of $\ep$.

\subsection{Control of the pressure}\label{subsect-pressure}
We have previously proved in Lemma \ref{lem:bounds-1} that $v_\ep$ was bounded from below (cf~\eqref{eq:lower-bound-vep}):
\[
v_\ep \geq 1 + C_1 \ep^{\frac{1}{\gamma-1}}.
\]
Unfortunately, this bound does not provide any control on the pressure $p_\ep(v_\ep)$ as $\ep \to 0$ since it only yields the inequality
\[
p_\ep(v_\ep) 
\leq \dfrac{\ep}{\big(1+C_1\ep^{\frac{1}{\gamma-1}} -1\big)^{\gamma}}
\lesssim \ep^{-\frac{1}{\gamma-1}}.
\] 
The goal of this section is to prove a uniform control of $\|p_\ep(v_\ep)\|_{L^1_{loc}}$.

\begin{prop}
Let Assumption \ref{ass-I-setup}-\ref{ass-II-setup} hold.
Then there exists a positive constant $C$, independent of $\ep$, such that
\begin{equation}
\|p_\ep(v_\ep)\|_{L^1((0,T) \times (-L,L))} \leq C \qquad \forall \ L > 0.
\end{equation} 
\end{prop}

\bigskip

\begin{proof}


Thanks to Assumption \ref{ass-II-setup}, there exists two positive constants $\ell^*>0$ and $\underline{v}> 1$, independent of $\ep$, such that
\begin{align}\label{eq:average-v0}
< v_\ep^0 > \ := \ \dfrac{1}{2\ell} \int_{-\ell}^\ell v_\ep^0(x) \, \d x \geq \underline{v} > 1 \qquad \forall \ \ell \geq \ell^*.
\end{align}
From the first equation of system \eqref{eq:psystem-ep}, we infer that
\begin{align}\label{eq:average-v}
<v_\ep(t)> \ 
:= \ \dfrac{1}{2\ell} \int_{-\ell}^\ell v_\ep(t,x) \, \d x
=\ <v_\ep^0> + \dfrac{1}{2\ell}\int_0^t(u_\ep(s,\ell)-u_\ep(s,-\ell))\, \d s.
\end{align}
From the $L^\infty_x$ bound on $u_\ep$ provided by \eqref{eq:u-Linfty} of Lemma \ref{lem:bounds-1}, we have
\[
u_\ep(t, \ell)-u_\ep(t, -\ell) 
\geq - 2 \|u_\ep\|_{L^\infty} = - 2C_2 \quad \forall \ t \in [0,T],
\]
so that
\begin{align}\label{eq:bound-average}
<v_\ep(t)> &= \ <v_\ep^0>+ \dfrac{1}{2\ell}\int_0^t u_\ep(s, \ell)-u_\ep(s, -\ell) \, \d s \notag\\
& \ge \; <v_\ep^0> - \dfrac{C_2 t}{\ell} \notag\\
& \ge \underline{v}- \dfrac{C_2 t}{\ell}. \notag\\
\end{align}
Choosing now
\begin{equation}\label{df:L*}
\ell \geq \max \left\{ \ell^*, \dfrac{2C_2T}{\underline{v} -1}\right\}:= L^*,
\end{equation}
for all $t \in [0,T]$ one has that
\begin{align} \label{eq:low-<v>}
<v_\ep(t)> 
& \geq \underline{v} - \dfrac{C_2 T}{\ell} \geq \underline{v} + \dfrac{1 - \underline{v}}{2}  \geq \dfrac{\underline{v}+1}{2} > 1.
\end{align}

%
%
%

\bigskip
In order to control of the pressure, let us define the function
\[
\phi(t,x) 
= \begin{cases}
\displaystyle \dfrac{(x+L)}{2L} \int_{-L}^{L} v_\ep(t,z)\d z -  \int_{-L}^{x} v_\ep(t,z)\d z
\qquad &  \text{if} \quad x \in [-L, L], \\
0 \quad &  \text{otherwise},
\end{cases}
\]
for some fixed $L \geq L^*$ where $L^*$ has been introduced in \eqref{df:L*} and is independent of $\ep$.
Since $v_\ep$ is smooth, then $\phi \in \mathcal{C}^1_c([0,T]\times \R)$.
Now, we multiply the momentum equation in \eqref{eq:psystem-ep} (which holds point-wisely) by $\phi$ and we integrate in space and time. This yields
\[
\int_0^T \int_{-L}^{L} p_\ep(v_\ep) \partial_x \phi \, \d x \d t
= - \int_0^T \int_{-L}^{L} u_\ep \partial_t \phi \, \d x \d t 
+ \int_{-L}^{L} u_\ep^0(x) \phi(0,x) \, \d x .
\]
In the right-hand side first we have
\begin{align*}
& \int_0^T \int_{-L}^{L} u_\ep(t,x) \partial_t \phi(t,x) \d x \d t \\
& = \int_0^T \int_{-L}^{L} u_\ep(t,x) \Bigg[\dfrac{(x+L)}{2L} \int_{-L}^{L}\partial_t v_\ep(t,z)\d z -  \int_{-L}^{x} \partial_t v_\ep(t,z)\d z \Bigg] \d x \d t \\
& = \int_0^T \int_{-L}^{L} u_\ep(t,x) \Bigg[\dfrac{(x+L)}{2L} \int_{-L}^{L}\partial_x u_\ep(t,z)\d z -  \int_{-L}^{x} \partial_x u_\ep(t,z)\d z \Bigg] \d x \d t \\
& = \int_0^T \int_{-L}^{L} u_\ep(t,x) \Bigg[\dfrac{(x+L)}{2L} \Big(u_\ep(t,L) - u_\ep(t,- L)\Big) -  \Big(u_\ep(t,x) - u_\ep(t,- L)\Big) \Bigg] \d x \d t.
\end{align*}
Hence
\begin{align*}
 \left|\int_0^T \int_{-L}^{L} u_\ep(t,x) \partial_t \phi(t,x) \d x \d t\right|
& \leq C(T,L) \|u_\ep\|_{L^\infty}.
\end{align*}
The next term is 
\begin{align*}
\int_{-L}^{L} u_\ep^0(x) \phi(0,x) \d x
& = \int_{-L}^{L} u_\ep^0(x) \Bigg[\dfrac{(x+L)}{2L} \int_{-L}^{L} v_\ep^0(z)\d z -  \int_{-L}^{x} v_\ep^0(z)\d z \Bigg] \, \d x,
\end{align*}
which is controlled as follows
\begin{align*}
\left|\int_{-L}^{L} u_\ep^0(x) \phi(0,x) \d x \right|
& \leq C(L) \|u_\ep^0\|_{L^\infty} \|v_\ep^0\|_{L^\infty}.
\end{align*}
Now, we split in two parts the integral involving the pressure as 
\begin{align*}
\int_0^T \int_{-L}^{L} p_\ep(v_\ep) \partial_x \phi \, \d x \d t 
& = \int_0^T \int_{-L}^{L} p_\ep(v_\ep) \partial_x \phi \, \mathbf{1}_{\{v_\ep > \frac{3+ \underline{v}}{4}\}} \, \d x \d t \\
& \quad + \int_0^T \int_{-L}^{L} p_\ep(v_\ep) \partial_x \phi \, \mathbf{1}_{\{v_\ep \leq \frac{3+ \underline{v}}{4}\}} \, \d x \d t.
\end{align*}

Since $\underline{v}>1$ uniformly in $\ep$ and
\[
1 < \frac{3+ \underline{v}}{4} = \dfrac{1 + \frac{\underline{v}+1}{2}}{2} < \frac{\underline{v}+1}{2} < \underline{v},
\]
then the pressure $p_\ep(v_\ep)$ remains bounded on the set $\{v_\ep > \frac{3+ \underline{v}}{4}\}$, so providing a positive constant $C$, independent of $\ep$, such that
\[
\left|\int_0^T \int_{-L}^{L} p_\ep(v_\ep) \partial_x \phi \, \mathbf{1}_{\{v_\ep > \frac{3+ \underline{v}}{4}\}} \, \d x \d t \right|
\leq C.
\]
Therefore, using \eqref{eq:low-<v>} and the fact that $L\ge L^*$ in \eqref{df:L*}, we have 
\begin{align*}
C & \geq \left|\int_0^T \int_{-L}^{L} p_\ep(v_\ep) \partial_x \phi \, \mathbf{1}_{\{v_\ep \leq \frac{3+ \underline{v}}{4}\}} \, \d x \d t \right| \\
& =  \left|\int_0^T \int_{-L}^{L} p_\ep(v_\ep) \left(\frac{1}{2L} \int_{-L}^{L} v_\ep(t,z) \, \d z - v_\ep(t,x)\right) \, \mathbf{1}_{\{v_\ep \leq \frac{3+ \underline{v}}{4}\}} \, \d x \d t \right| \\
& \geq \int_0^T \int_{-L}^{L} p_\ep(v_\ep) \left(\frac{\underline{v}+1}{2} - \frac{3+ \underline{v}}{4}\right) \, \mathbf{1}_{\{v_\ep \leq \frac{3+ \underline{v}}{4}\}} \, \d x \d t \\
& \geq  \dfrac{ \underline{v} - 1}{4} \int_0^T \int_{-L}^{L} p_\ep(v_\ep) \mathbf{1}_{\{v_\ep \leq \frac{3+ \underline{v}}{4}\}} \d x \d t.
\end{align*}
Thus we also obtain the control of the integral of the singular pressure in the region close to the singularity.
In the end, we proved that
\[
(p_\ep(v_\ep))_\ep ~\text{is bounded in}~L^1((0,T)\times(-L,L)) \quad \text{for all}~ L \geq L^*,
\]
and thus
\[
(p_\ep(v_\ep))_\ep ~\text{is bounded in}~L^1((0,T)\times(-L,L)) \quad \text{for all}~ L > 0.
\]
\end{proof}


\subsection{Passing to the limit as $\ep \rightarrow 0$} {\label{sec:epto0}}
In this section, we prove Theorem \ref{thm:main-smooth} under Assumptions \ref{ass-I-setup}
-\ref{ass-II-setup}.
On the time interval $[0,T]$, the solution $(v_\ep,u_\ep)$ exists and is regular.
More precisely, the previous sections have shown that there exists a constant $C$, independent of $\ep$ such that 
\begin{align*}
& \|u_\ep\|_{L^\infty((0,T) \times \R)} \le C,\\
& \|v_\ep\|_{L^\infty((0,T) \times \R)} \le C, \\
& \|\partial_x u_\ep\|_{L^\infty((0,T) \times \R)} \le C, \\
& \|\partial_x v_\ep\|_{L^\infty((0,T) \times \R)} \le C, \\
& \|p_\ep(v_\ep)\|_{L^1((0,T) \times (-L,L))} \le C, \quad \forall \ L > 0.
\end{align*}
From these bounds, we infer the existence of a pair $(v,u)$ such that
\begin{align*}
v_\ep \rightharpoonup v, \quad & \text{weakly-* in} \quad L^\infty((0,T);W^{1,\infty}(\R)), 
\\
u_\ep \rightharpoonup u \quad &\text{weakly-* in} \quad L^\infty((0,T);W^{1,\infty}(\R)).
\end{align*}
Since the lower bound \eqref{eq:lower-bound-vep} holds at $\ep$ fixed thanks to Lemma \ref{lem:bounds-1}, then the congestion constraint involving the specific volume is satisfied in the limit, i.e.
\begin{equation}
v(t,x) \geq 1 \quad \text{a.e.} ~ (t,x) \in (0,T) \times \R.
\end{equation}
Employing the bound on $\partial_x u_\ep$ in \eqref{eq:bound-dx-v} in the mass equation, \eqref{eq:psystem-ep-mass}, we uniformly control the time derivative of $v_\ep$ as follows
\[
\|\partial_t v_\ep\|_{L^\infty((0,T)\times \R)} = \|\partial_x u_\ep\|_{L^\infty((0,T)\times \R)} \leq C.
\]
This way, it is now an easy task to apply the classical Aubin-Lions Lemma to get
\begin{equation}\label{eq:conv-v}
v_\ep \rightarrow v \quad \text{in    }  C([0,T]\times [-L,L]), \quad \forall \ L > 0.
\end{equation}

\medskip
Applying the same reasoning to the velocity $u_\ep$, and using in particular the control of the $L^1_{t,x}$-norm of the pressure, we obtain
\[
\|\partial_t u_\ep\|_{L^1((0,T); W^{-1,1}_{loc}(\R))} \leq C.
\]
A generalization of the Aubin-Lions Lemma 
due to Simon, see \cite{simon1986}, provides the convergence
\begin{equation}\label{eq:conv-u}
u_\ep \rightarrow u \quad \text{in    }  L^q((0,T); \mathcal{C}([-L,L])), \quad \forall  \ q \in [1, +\infty), \ L > 0.
\end{equation}
%
%
As the pressure is made of two parts
\[
p_\ep(v_\ep) = \dfrac{\ep}{(v_\ep -1)^\gamma} + \dfrac{\kappa}{v_\ep^{\tgamma}} = p_{\ep, 1}(v_\ep) + p_{\ep, 2}(v_\ep),
\]
the strong convergence of the non-singular component $p_{\ep,2}(v_\ep)$ directly follows from \eqref{eq:conv-v}, so that
\[
p_{\ep, 2}(v_\ep) \rightarrow p_2(v) = \dfrac{\kappa}{v^{\tgamma}} \quad \text{in} \quad   C([0,T], W^{s,\infty}_{loc}(\R)), \quad \forall \ 0 < s < 1.
\]
About the singular component $p_{\ep,1}(v_\ep)$, we use the $L^1$ uniform bound to infer that
\begin{equation}
p_{\ep,1}(v_\ep) \rightharpoonup p \quad \text{in} \quad \mathcal{M}_+((0,T) \times (-L,L)) \quad \forall \ L >0.
\end{equation}
Finally, to recover the exclusion constraint, we pass to the limit in the equality
\[
(v_\ep - 1)p_{\ep,1}(v_\ep)
= \dfrac{\ep}{(v_\ep-1)^{\gamma-1}}.
\]
On the one hand $(v_\ep-1)$ converges in $\mathcal{C}([0,T]\times [-L,L])$ to $v-1$, so that the left-hand side of the above equality converges in the sense of distribution towards $(v-1)p$.
While the right-hand side converges strongly to $0$ in $L^{\frac{\gamma}{\gamma-1}}((0,T) \times (-L,L))$ thanks to the following inequality (and the uniform $L^1$ bound on the pressure),
\[
\dfrac{\ep}{(v_\ep-1)^{\gamma-1}}  
\leq \ep^{\frac{1}{\gamma}} \left( \dfrac{\ep}{(v_\ep -1)^\gamma} \right)^{\frac{\gamma-1}{\gamma}} = \ep^{\frac{1}{\gamma}} \left( p^1_\ep(v_\ep) \right)^{\frac{\gamma-1}{\gamma}}.
\]
Hence, we get the desired exclusion constraint
\begin{equation}
(v-1) \ p = 0,
\end{equation}
which finally allows to show that $(v,u,p)$ is a weak solution of the free-congested Euler equations. 

\medskip
As a final remark, note that we able to show that the incompressibility constraint is satisfied by the limit velocity $u$ in the congested domain where $v=1$.
We have indeed the following lemma whose proof is postponed to the Appendix
\begin{lemma}
Let $T>0$, $v\in W^{1, \infty}((0,T) \times \R)$ and $u \in L^\infty((0,T); W^{1,\infty}(\R))$ satisfying
\begin{equation*}
\partial_t v = \partial_x u , \quad v_{|t=0} = v^0 \quad a.e.
\end{equation*}
The following two assertions are equivalent:
\begin{enumerate}
    \item $v(t,x) \geq 1$ for all $(t,x) \in [0,T] \times \R$;
    \item $\partial_x u = 0$ a.e. on $\{v \leq 1\}$ and $v^0 \geq 1$.
\end{enumerate}
\end{lemma}

\medskip

\section{Appendix}
This Appendix is dedicated to the proof of several technical results.

\subsection{Invariant region for the viscous system}{\label{sec:inv-reg}}

Let us first recall some definitions from Dafermos book \cite{dafermos2000}.

\begin{df}[Invariant region]
	A closed subset $\Sigma \in \mathbb{R}^2$ is called a (positively) invariant region for the local solution of (\ref{eq:sys-v-mu}) defined on $[0,\tau)$, if for any initial data $U_0$ such that $U^0(x)\in\Sigma$ for all $x\in \R$, it holds that $U(t,x) \in \Sigma$ for all $t \in [0, \tau)$, $x \in \R$.
\end{df}

\begin{df}\label{def:quasi-convex-funct2}
	Let $\mathcal{U}$ be a convex subset of $\R^2$. 
	A function $G: \mathcal{U} \rightarrow \mathbb{R}$ is \emph{quasi-convex} if, for any $U, V \in \mathcal{U}$ and $\nu \in [0,1]$, we have:
	$$G(\nu U + (1-\nu) V) \le \max \{ G(U), G(V) \}.$$
\end{df}

In other words, a function $G$ is called \emph{quasi-convex} if, for any $c \in \mathbb{R}$, its sublevel sets $\{G(U) \le c\}$ are convex. Notice, indeed, that a function whose sublevel sets are convex sets may fail to be a convex function.

The following theorem is proved in \cites{smoller2012, chueh1977}.
\begin{thm}\label{thm:invariant-region}
	Define $\Sigma=\cap_{i=1}^m \{U \in \mathbb{R}^2 \, : \; G_i(U) \le 0\}$ for a finite collection of smooth functions $G_i(U)$, $i=1, \dots ,m$. 
	Suppose that for any boundary point $U_{\partial \Sigma} \in \partial \Sigma$, denoting by $dG_i$ the derivative of $G_i$ at the boundary point $U_{\partial \Sigma}$, the following conditions are satisfied:
	\begin{itemize}
		\item  $dG_i$ is a left eigenvector of $A=df_\ep$ for $i=1, \cdots, m$;
		\item  $G_i$ is quasi-convex at $U_{\partial \Sigma}$.
	\end{itemize}
	Then, $\Sigma$ is a positively invariant region for (\ref{eq:sys-v-mu}).
\end{thm}

We present here the main idea behind Proposition \ref{prop:invariant}, before providing its proof below.

\bigskip
\begin{rmk}
	Multiplying Equations \eqref{eq:viscous} by $(\partial_\rho w, \partial_m w)$, one can check that $(w,z)$ satisfy
	\begin{align*}
	\partial_t w + \lambda_2 \partial_x w 
	\leq \mu \partial_{xx} w + 2\mu \dfrac{\partial_x\rho_{\mu}}{\rho_{\mu}}\partial_x w,
	\end{align*}
	\begin{align*}
	\partial_t z + \lambda_1 \partial_x z 
	\ge \mu \partial_{xx} z + 2\mu \dfrac{\partial_x\rho_{\mu}}{\rho_{\mu}}\partial_x z,
	\end{align*}
	with characteristics $\lambda_{1,2}$ given by
	\begin{equation*}
	\lambda_1 = \dfrac{m_{\mu}}{\rho_{\mu}} - \sqrt{p'_\ep(\rho_{\mu})}, \quad 
	\lambda_2 = \dfrac{m_{\mu}}{\rho_{\mu}} + \sqrt{p'_\ep(\rho_{\mu})}.
	\end{equation*}
	Hence, formally, a maximum principle would imply that
	$$w(t, x) \le \|w(0,\cdot)\|_{L_x^\infty}, \qquad z(t, x) \ge -  \|z(0,\cdot)\|_{L_x^\infty} $$
	and since $z \leq w$ a.e., we would infer that
	\[
	z(t, x) \geq  -  \|w(0,\cdot)\|_{L_x^\infty}.
	\]
	In the next proof, we show that these inequalities actually hold, as an application of Theorem~\ref{thm:invariant-region}.
\end{rmk}

\bigskip
\begin{proof}[Proof of Proposition \ref{prop:invariant}]
	We define the functions $G_{1,2}$ as
	\[
	G_1(U)=w-M, \quad G_2(U)=-z-M,
	\]
	in such a way that $\Sigma$ reads
	\begin{equation*}
	\Sigma = \{U \in \mathbb{R}^2 \, : \; G_1(U) \le 0, \; G_2(U)\leq 0\}.
	\end{equation*}
	Before going further, let us show that $M$ can be bounded uniformly with respect to $\mu$ and $\ep$.
	First, from the definition of $w$ \eqref{df:w} and assumptions \eqref{eq:alpha0-ep}-\eqref{eq:m0-ep} on the initial data, we have
	\[
	M 
	\leq A^0_\ep + \int_0^{A^0_\ep} \dfrac{\sqrt{p'_\ep(s)}}{s} ds,
	\]
	which proves that $M$ is bounded uniformly with respect to the artificial viscosity $\mu$.
	To show that $M$ is also controlled uniformly with respect to $\ep$, we need to control the integral involving the singular pressure $p_\ep$.
	For that purpose, we introduce $\Theta_\ep$ the primitive of $s \mapsto \sqrt{p'_\ep(s)}/s$ vanishing at $s=0$.
	We have then
	\[
	M \leq  A^0_\ep + \Theta_\ep(A^0_\ep),
	\]
	where 
	\[
	\Theta_\ep(s) 
	\underset{s \to 1^-}{\sim} C_\gamma \dfrac{\sqrt{\ep}}{(1-s)^{\frac{\gamma-1}{2}}}
	\underset{s \to 1^-}{\sim} \tilde{C}_\gamma \sqrt{H_\ep(s)}
	\]
	for some positive constants $C_\gamma,\ \tilde{C}_\gamma$ depending only on $\gamma$. 
	By assumption \eqref{eq:alpha0-ep}, which provides the control of the initial energy \eqref{eq:bound-H0} (see Remark \ref{rmk:initial}) the sequence $(\Theta_\ep(A^0_\ep))_\ep$ is bounded by some positive constant independent of $\ep$ (and $\mu$).
	As a consequence, $M$ is controlled uniformly with respect to $\ep$.
	
	\medskip
	To show that the region $\Sigma$ is invariant, we have to check the validity of the hypotheses of Theorem \ref{thm:invariant-region}. 
	Actually, the only assumption to be verified is the quasi-convexity of $G_1(U), \, G_2(U)$ at the boundary $\partial \Sigma$ given by 
	\begin{align*}
	\partial \Sigma
	& = \{U \in \R^2 \, | \, G_1(U)=0, \, G_2(U) \leq 0 \}  
	\ \bigcup \ \{U \in \R^2 \, | \ G_1(U) \leq 0, \, \, G_2(U)=0 \} \\
	& =: \partial \Sigma_1 \cup \partial \Sigma_2 .
	\end{align*}	
	In other words, at the boundary we have either $w = M$ or $z = - M$, which means that
	\begin{equation}
	m = L(\rho) \quad \text{or} \quad 
	m = - L(\rho)  \quad \forall \ (\rho,m) \in \partial \Sigma
	\end{equation}	
	where, by definition of $w$ and $z$, the function $L$ is
	\begin{equation*}
	L(\rho) :=  \rho M - \rho \int_0^\rho{\dfrac{\sqrt{p'_\ep(s)}}{s} \ \d s}. 
	\end{equation*}
	Notice that
	\begin{align}
	&  G_1(U) \leq 0  \quad \Longleftrightarrow \quad L(\rho) \geq m \label{eq:G1L1},\\
	&  G_2(U) \leq 0  \quad \Longleftrightarrow \quad  -L(\rho) \leq m \nonumber.
	\end{align}
	By Definition \ref{def:quasi-convex-funct2}, we need to check that for all $U^a, U^b \in \partial \Sigma $ and $\nu \in [0,1]$, we have:
	\begin{align}
	& G_1(\nu U^a + (1-\nu) U^b) \le \max \{ G_1(U^a), G_1(U^b) \} , \label{cond-qc1}\\
	& G_2(\nu U^a + (1-\nu) U^b) \le \max \{ G_2(U^a), G_2(U^b) \} \label{cond-qc2}.
	\end{align}
	Let us check condition \eqref{cond-qc1}.
	First, let us remark that the function $L$ is concave. Indeed
	\begin{align*}
	L'(\rho) 
	& = M - \int_0^\rho{\dfrac{\sqrt{p'_\ep(s)}}{s} \ \d s} - \sqrt{p'_\ep(\rho)}
	\end{align*}
	and thus, recalling that $\gamma > 1$,
	\begin{align*}
	L''(\rho) 
	& = - \dfrac{\sqrt{p'_\ep(\rho)}}{\rho} - \dfrac{p''_\ep(\rho)}{2 \sqrt{p'_\ep(\rho)}} \\
	& = - \sqrt{\ep \gamma} \left[\dfrac{\rho^{\frac{\gamma-3}{2}}}{(1-\rho)^{\frac{\gamma+1}{2}}}
	+ \dfrac{\gamma-1 +2\rho}{2} \dfrac{\rho^{\frac{\gamma-3}{2}}}{(1-\rho)^{\frac{\gamma+3}{2}}}
	\right]
	< 0.
	\end{align*}
	As a consequence of the concavity of $L$, we deduce that for any $\nu \in [0,1]$,
	\begin{align}\label{eq:L1}
	\nu L (\rho^a) + (1-\nu) L(\rho^b)
	& \leq L(\nu \rho^a + (1-\nu) \rho^b).
	\end{align}
	Let us now distinguish the two cases:
	\begin{itemize}
		\item $U^a$ and/or $U^b$ belong to $\partial \Sigma_1$, for instance $G_1(U^a) = 0$ and $G_1(U^b) \leq 0$. 
		Then inequality \eqref{eq:L1} yields
		\begin{align*}
		\nu m^a + (1-\nu) m^b
		& \leq L(\nu \rho^a + (1-\nu) \rho^b) 
		\end{align*}
		By characterization \eqref{eq:G1L1}, this means that
		\[
		G_1(\nu U^a + (1-\nu) U^b) \leq 0 = \max \{ G_1(U^a), G_1(U^b) \}.
		\]
		
		\item Both $U^a$, $U^b$ belong to $\partial \Sigma_2 \setminus \partial \Sigma_1$ with for instance 
		$G_1(U^a) < G_1(U^b) < 0$, we have to show that
		\[
		G_1(\nu U^a + (1-\nu) U^b)
		= G_1(U^c)
		\leq G_1(U^b) 
		= \frac{1}{\rho^b}(m^b -L(\rho^b)) 
		= 2 \frac{m^b}{\rho^b}
		\]
		where for the last equality we have used the fact that $G_2(U^b) = 0 \implies m^b = -L(\rho^b)$. 
		This yields
		\begin{align*}
		G_1(U^c)
		& = \dfrac{1}{\rho^c} \left(m^c -L(\rho^c)\right) \\
		& \leq \dfrac{1}{\rho^c} \left(m^c -\nu L(\rho^a) - (1-\nu) L(\rho^b) \right) 
		\quad \text{by concavity of} ~ L \\
		& \quad = \dfrac{1}{\rho^c} \left(m^c + \nu m^a + (1-\nu) m^b \right) 
		\quad \text{since} ~ G_2(U^a) = G_2(U^b) = 0.
		\end{align*}
		Now, since we have assumed that $G_1(U^a) < G_1(U^b)$ (and $G_2(U^a) = G_2(U^b) = 0$), then from \eqref{eq:G1L1} the following inequality follows
		\[
		\frac{m^a}{\rho^a} < \frac{m^b}{\rho^b}.
		\]
		As a consequence of this inequality, we deduce that
		\begin{align*}
		G_1(U^c) 
		& \leq \dfrac{2}{\rho^c} \left(\nu m^a + (1-\nu) m^b \right) \\
		& \leq \dfrac{2}{\rho^c} \left(\nu \frac{\rho^a}{\rho^b}m^b - (1-\nu) m^b \right) \\
		& \leq \dfrac{2}{\rho^c} \frac{m^b}{\rho^b}\left(\nu \rho^a - (1-\nu) \rho^b \right) \\
		& \leq 2 \frac{m^b}{\rho^b}.
		\end{align*}
	\end{itemize} 
	Finally we obtained in both cases inequality \eqref{cond-qc1}, that is
	\[
	G_1(\nu U^a + (1-\nu) U^b) \le \max \{ G_1(U^a), G_1(U^b) \}  \quad \forall \ (U^a, U^b) \in \partial \Sigma.
	\]
	Inequality \eqref{cond-qc2} can be proved with the same arguments.
	Therefore the assumptions of Theorem \ref{thm:invariant-region} are satisfied and the set $\Sigma$ defined in \eqref{df:Sigma}
	results to be invariant.
\end{proof}

\subsection{Reduction of the Young measure}
	The aim of this part of the appendix is to show the details of the computations which allow to prove that the support of the Young measure associated to the vanishing viscosity weak limits $(\brho_\ep, \bm_\ep)$ of $(\rho_{\ep, \mu}, m_{\ep, \mu})$ is reduced to a point, then yielding the strong convergence. \\
	To this end, we revisit a strategy which can be found for instance in Lu's book [Section 8.4, \cite{lu2002}] and was specifically designed for the case of the isentropic Euler equations with the adiabatic pressure law $p(\rho) = \kappa \rho^\gamma$, $\gamma \in (1,3]$. 
	We show here that we are able to extend this result to the singular pressure $p_\ep(\rho) = \ep \left({\rho}/{(1-\rho)}\right)^\gamma$, again with $\gamma \in (1,3]$, by adding a new idea in the previous strategy.
	This novelty and the generality of our computations could be useful in other cases.
	

	\medskip
	We recall the statement of Lemma \ref{lem:combination} that we are going to prove in this appendix.
	\begin{lemma}{\label{lem:calcul-app}}
		Assume that $ \gamma \in (1,3]$. Then the following equality holds
\begin{align} \label{eq:HOHO-app}
& \ep \dfrac{3-\gamma}{2(\gamma+1)} \dfrac{\brho^{\gamma+1}}{(1-\brho)^\gamma} < \nu, (u -\bu)^4 >
+ \ep^3 \dfrac{\gamma^2(5\gamma +1)}{2(\gamma+1)}  \dfrac{\brho^{3\gamma-5}}{(1-\brho)^{3\gamma+2}} < \nu, (\rho -\brho)^4 >  \nonumber \\
& + \ep^2 \ 6 \gamma  \dfrac{\brho^{2(\gamma-1)}}{(1-\brho)^{2\gamma +1}} \big(< \nu, (u-\bu) (\rho -\brho)> \big)^2  + \ \ER
= 0
\end{align}
where $\ER$ denotes ``an error'' whose $L^\infty$ norm is negligible compared to the other terms of the equality.
	\end{lemma}
	
	\bigskip
	The proof of this lemma is technical and will be split therefore in different steps.
	\begin{itemize}
		\item Some algebraic expressions obtained as combinations of the outcomes of the Div-Curl Lemma in \ref{lem:div-curl} are established.
		\item The above-mentioned algebraic relations will then be employed to get inequality \eqref{eq:HOHO-app} which allows, after separation of the two cases $\gamma < 3$ and $\gamma=3$, to reduce the support of the measure $\nu$  to the point $(\bar{\rho},\bar{m})$.
	\end{itemize}

	\bigskip
	To simplify some further computations, we first Taylor-expand the most complicated ``relative'' entropy-entropy flux pairs $(\teta_i,\tq_i)$.
	\begin{lemma}{\label{lem:exp-eta-q}}
		The pairs $(\teta_i,\tq_i)$, $i=3,4$ can be rewritten as follows
		\begin{align*}
		\teta_3 & = \dfrac{\brho}{2} (u-\bu)^2 + \dfrac{1}{2} \dfrac{p'_\ep(\brho)}{\brho} (\rho-\brho)^2 + O_3\\
		\tq_3 & = \dfrac{1}{2} \brho \bu (u-\bu)^2 
		+ \dfrac{1}{2}\dfrac{p'_\ep(\brho)}{\brho} \bu (\rho -\brho)^2 
		+ p'_\ep(\brho) (u-\bu)(\rho-\brho) + O_3 \\
		\teta_4 & = 3\brho \bu (u- \bu)^2 
		+ 3\dfrac{p'_\ep(\brho)}{\brho}(\rho-\brho)^2 
		+ 6 \dfrac{p_\ep(\brho)}{\brho} (u-\bu)(\rho -\brho) + O_3 \\
		\tq_4 & = 3\big(p_\ep(\brho) + \brho \bu^2\big) (u-\bu)^2 
		+ 3\dfrac{p'_\ep(\brho)}{\brho^2}\big(p_\ep(\brho) + \brho \bu^2\big) (\rho-\brho)^2 \nonumber \\
		& \quad + 6\left(\dfrac{p_\ep(\brho)}{\brho} + p'_\ep(\brho)\right) \bu (u-\bu)(\rho -\brho) + O_3
		\end{align*}
		where 
		\[
		\|O_3\|_{L^\infty} \lesssim \sum_{k+l = 3} \|\big((\rho-\brho)^k (u-\bu)^l\big)\|_{L^\infty}
		\]
	\end{lemma}
	
	In the same manner, for sake of simplicity, we introduce the notation $\ER$ for terms of the form
	\[
	<\nu, \big((\rho-\brho)^k (u-\bu)^l\big)> \quad \text{with} \quad k+l = 5,
	\]
	and therefore
	\[
	\|\ER\|_{L^\infty} \lesssim  \sum_{k+l =5} \|\big((\rho-\brho)^k (u-\bu)^l\big)\|_{L^\infty}.
	\]

	\subsubsection*{Step 1 - Combinations of the identities provided by the Div-Curl Lemma}
	Here we prove two lemmas, which give two crucial algebraic relations to be used in the last part of the proof.
	
	\begin{lemma}
		{\label{lem:A2}}
		The following equality holds
		\begin{align} {\label{eq:lem1}}
		& A_1(\brho) < \nu, (\rho -\brho)^4>
		+ A_2(\brho) \big(< \nu, (\rho -\brho)^2> \big)^2  \nonumber \\
		& + A_3(\brho) \Big[ < \nu, (\rho -\brho)^2> \ < \nu, (u -\bu)^2>
		+ < \nu, (\rho -\brho)^2(u -\bu)^2> \Big]  \nonumber \\
		& + A_4(\brho)  \big(< \nu, (u -\bu)^2> \big)^2  \nonumber \\
		& 
		= \brho^2  < \nu, (u -\bu)^4> + \ \ER
		\end{align}
		with
		\begin{align*}
		A_1(\brho) & = \dfrac{2(p_\ep'(\brho))^2 - p_\ep(\brho)p_\ep''(\brho)}{2\brho^2} - \dfrac{p_\ep'(\brho)p_\ep(\brho)}{\brho^3};\\
		A_2(\brho) & = \dfrac{3p_\ep(\brho)p_\ep''(\brho)}{2\brho^2};\\
		A_3(\brho) & = \dfrac{3p_\ep(\brho)}{\brho} + \dfrac{3p_\ep(\brho)p_\ep''(\brho)}{2p_\ep'(\brho)};\\
		A_4(\brho) & = \dfrac{3\brho p_\ep(\brho)}{p_\ep'(\brho)}.
		\end{align*}
	\end{lemma}
	
	\bigskip
	\begin{proof}
		First of all, observing that
		\[
		<\nu, \teta_1 >\  = \ <\nu, \rho - \brho > \ = 0,\quad <\nu, \tq_1> \ = \ <\nu, m -\bm> \ = 0,
		\]
		we infer from the identities \eqref{eq:divcurl-id} derived from Div-Curl Lemma, that 
		\begin{equation}
		\begin{cases}
		\ <\nu, \teta_1 \tq_{3} - \teta_{3} \tq_1> \
		= \ <\nu, \teta_1 > \ < \nu, \tq_{3}> - <\nu,\teta_{3}> <\nu, \tq_1>
		= 0, \\
		\ <\nu, \teta_1 \tq_{4} - \teta_{4} \tq_1> \
		= \ <\nu, \teta_1 > \ < \nu, \tq_{4}> - <\nu,\teta_{4}> <\nu, \tq_1>
		= 0.
		\end{cases}
		\end{equation}
		Using the explicit expressions of  $(\teta_i,\tq_i)$, after some computations we end up with the two identities
		\begin{align} {\label{eq:v1v3}}
		0
		& = <\nu, \teta_1 \tq_{3} - \teta_{3} \tq_1> \nonumber \\
		& = <\nu, (\rho - \brho)\Bigg[\frac{1}{2}m \left(u - \bu \right)^2
		+ \left(u - \bu\right)(p_\ep(\rho) - p_\ep(\brho)) \nonumber\\
		& \hspace{2.5cm} 
		+ u\left(\rho \int_{\brho}^{\rho}\frac{p_\ep(s)}{s^2} - \frac{p_\ep(\brho)}{\brho}(\rho-\brho) \right)\Bigg]> \nonumber\\
		& \quad - <\nu, (m-\bm) \Bigg[\frac{1}{2}\rho \left(u - \bu \right)^2 +
		\rho \int^\rho \frac{p_\ep(s)}{s^2}  - \brho \int^{\brho} \frac{p_\ep(s)}{s^2} \nonumber \\
		& \hspace{3cm}
		- \left( \int^{\brho} \frac{p_\ep(s)}{s^2}  + \frac{p_\ep(\brho)}{\brho} \right) (\rho - \brho)\Bigg]> \nonumber \\
		& = <\nu, \left[p_\ep(\rho)(\rho - \brho) - \rho \brho \left(\int_{\brho}^{\rho}\dfrac{p_\ep(s)}{s^2}ds\right)\right](u-\bu)> 
		- < \nu, \dfrac{1}{2} \rho \brho (u - \bu)^3>,
		\end{align}
		and
		\begin{align}{\label{eq:v1v4}}
		0
		& = <\nu, \teta_1 \tq_{4} - \teta_{4} \tq_1> \nonumber \\
		& = 3 < \nu, (\rho -\brho)\left[2\Delta + (u^2-\bu^2)p_\ep(\rho)\right]>
		- < \nu, \rho \brho(u-\bu)^3(u + 2\bu)> \\
		& \quad	 - 6 < \nu, \rho \brho u(u-\bu) \int_{\brho}^{\rho} \dfrac{p_\ep(s)}{s^2} ds > \nonumber ,
		\end{align}
		where
		\[
		\Delta 
		= p_\ep(\rho)\int_{\brho}^{\rho} \dfrac{p_\ep(s)}{s^2} ds 
		- \int_{\brho}^{\rho} \dfrac{(p_\ep(s))^2}{s^2} ds.
		\]
		The combination $\displaystyle \eqref{eq:v1v4}-6\bu\times \eqref{eq:v1v3}$ then yields 
		\begin{align*}
		0 & = 3 < \nu, (\rho -\brho)\left[2\Delta + (u^2-\bu^2)p_\ep(\rho)\right]>
		- < \nu, \rho \brho(u-\bu)^3(u + 2\bu)> \\
		& \quad	 - 6 < \nu, \rho \brho u(u-\bu) \int_{\brho}^{\rho} \dfrac{p_\ep(s)}{s^2} ds >
		-6 <\nu, p_\ep(\rho)(\rho - \brho) (u-\bu)\bu > \nonumber \\ 
		& \quad  + 6 <\nu,\rho \brho \left(\int_{\brho}^{\rho}\dfrac{p_\ep(s)}{s^2}ds\right)(u-\bu)\bu >
		+ 6 < \nu, \dfrac{1}{2} \rho \brho \bu (u - \bu)^3> \nonumber \\
		& = 6 <\nu,(\rho-\brho)\Delta>
		- 6 < \nu, \rho \brho\left(\int_{\brho}^{\rho}\dfrac{p_\ep(s)}{s^2}ds \right)(u-\bu)^2> \\
		& \quad + 3<\nu,(\rho-\brho)p_\ep(\rho)(u-\bu)^2> 
		-  <\nu, \rho \brho(u-\bu)^4>.
		\end{align*}
		Performing a Taylor expansion of the right-hand side (in particular we expand $\Delta$ up to order 3), and recalling that we denote
		\[
		\ER  = \sum_{k+l = 5} <\nu, O\big((\rho-\brho)^k (u-\bu)^l\big)>,
		\]
		we obtain
		\begin{align}\label{eq:lem11}
		0 
		& = \left(\dfrac{2p''_\ep(\brho)p_\ep(\brho) +(p'_\ep(\brho))^2}{\brho^2} - 2\dfrac{p'_\ep(\brho)p_\ep(\brho)}{\brho^3} \right) < \nu, (\rho-\brho)^4>  \\
		& \quad	+ 3 \dfrac{p'_\ep(\brho)p_\ep(\brho)}{\brho^2} < \nu, (\rho-\brho)^3> 
		- 3 p_\ep(\brho) <\nu,(\rho-\brho)(u-\bu)^2 > \nonumber \\
		& \quad 
		- \brho^2 <\nu,(u-\bu)^4 >
		+ \ER . \nonumber
		\end{align}
		
		\bigskip
		Coming back to the following outcome of the Div-Curl Lemma 
		\[
		<\nu, \teta_2 \tq_{3} - \teta_{3} \tq_2> \
		= \ <\nu, \teta_2 > \ < \nu, \tq_{3}> - <\nu, \teta_3 > \ < \nu, \tq_{2}>,
		\]
		and noticing that $<\nu, \teta_2 > = 0$, one gets
		\[
		0
		= <\nu, \teta_2 \tq_{3} - \teta_{3} \tq_2> + <\nu,\teta_{3}> <\nu, \tq_2>. 
		\]
		The use of the explicit expressions of $(\teta_i,\tq_i), \, i=2,3$, and a Taylor expansion as before yield
		\begin{align}\label{eq:lem12}
		0 
		& = \dfrac{2(p'_\ep(\brho))^2 -5\brho p'_\ep(\brho) p''_\ep(\brho)}{12\brho^2} <\nu,(\rho-\brho)^4> \\
		& \quad - \dfrac{(p'_\ep(\brho))^2}{2\brho} <\nu,(\rho-\brho)^3>
		+ \dfrac{1}{2} \brho p'_\ep(\brho) <\nu,(\rho-\brho)(u-\bu)^2> \nonumber \\
		& \quad + \left(\dfrac{p'_\ep(\brho)}{2} + \dfrac{\brho p''_\ep(\brho)}{4} \right) <\nu,(\rho-\brho)^2(u-\bu)^2> \nonumber \\
		& \quad + \left(\dfrac{p'_\ep(\brho)}{2} + \dfrac{\brho p''_\ep(\brho)}{4} \right) <\nu,(\rho-\brho)^2> <\nu,(u-\bu)^2> \nonumber \\
		& \quad + \dfrac{p''_\ep(\brho) p'_\ep(\brho)}{4\brho} \big(<\nu,(\rho-\brho)^2>\big)^2
		+ \dfrac{\brho^2}{2} \big( <\nu,(u-\bu)^2>  \big)^2
		+ \ER .\nonumber
		\end{align}
		We finally obtain the result \eqref{eq:lem1} of Lemma \ref{lem:A2} thanks to the combination 
		\[
		\eqref{eq:lem11} + \dfrac{6p_\ep(\brho)}{\brho p'_\ep(\brho)} \times \eqref{eq:lem12}.
		\]
	\end{proof}
	
	\bigskip
	
	\begin{lemma}{\label{lem:app-2}}
		The relation below holds true:
		\begin{align}{\label{eq:lem2}}
		&  B_1(\brho) < \nu,(u -\bu)^4> + B_2(\brho)<\nu,(\rho -\brho)^4 > \nonumber \\
		& \quad - B_3(\brho)\Big[<\nu,(u -\bu)^2(\rho - \brho)^2> + <\nu,(\rho-\brho)^2> <\nu,(u-\bu)^2>\Big]
		\nonumber\\
		& = B_1(\brho) \big(< \nu,(u -\bu)^2>\big)^2 
		+ B_2(\brho) \big(< \nu,(\rho -\brho)^2>\big)^2    \\
		& \quad	- 2 B_3(\brho) \big(< \nu,(\rho-\brho)(u-\bu)>\big)^2
		+ \ \ER \nonumber
		\end{align}	
		where
		\begin{align*}
		B_1(\brho) & = \dfrac{3\brho p_\ep(\brho)}{2},\\
		B_2(\brho) & = \dfrac{3 p_\ep(\brho) (p'_\ep(\brho))^2}{2\brho^3},\\
		B_3(\brho) & = \dfrac{3p_\ep(\brho) p'_\ep(\brho)}{\brho}.
		\end{align*}
	\end{lemma}

	\begin{proof}[Proof of Lemma \ref{lem:app-2}]
		After some easy computations, one can check that
		\begin{align*}
		& < \nu, \teta_3 \tq_4 - \teta_4 \tq_3> \\
		& = \ < \nu, \left(\frac{1}{2}\rho p_\ep(\rho) + \rho p_\ep(\brho)\right)(u-\bu)^4 >  \\
		& \quad + 6 < \nu, \Big[\rho\int_{\brho}^{\rho}\frac{p_\ep(s)}{s^2} - \frac{p_\ep(\brho)}{\brho} (\rho -\brho)\Big]
		\Big[p_\ep(\rho) \int_{\brho}^{\rho}\frac{p_\ep(s)}{s^2} - \int_{\brho}^{\rho}\left(\frac{p_\ep(s)}{s}\right)^2\Big] > \\
		& \quad + < \nu, 3(u-\bu)^2 
		\Big[-\rho \int_{\brho}^{\rho}\left(\frac{p_\ep(s)}{s}\right)^2 
		- (\rho-\brho)\frac{p_\ep(\brho)p_\ep(\rho)}{\brho} 
		+ 2\rho p_\ep(\brho) \int_{\brho}^{\rho}\frac{p_\ep(s)}{s^2}
		\Big]> \\
		& = < \nu, T_1 + T_2 + T_3 >.
		\end{align*}
		Expanding the different terms $T_i$ we get
		\begin{align*}
		T_1 
		& = \left(\frac{1}{2}\rho p_\ep(\rho) + \rho p_\ep(\brho)\right)(u-\bu)^4  \\
		& = \left(\frac{1}{2}\brho p_\ep(\brho) + \brho p_\ep(\brho)\right)(u-\bu)^4 + \ER \\
		& = \frac{3}{2}\brho p_\ep(\brho) (u-\bu)^4 + \ER. 
		\end{align*}
		We then write
		\begin{align*}
		T_2 
		& =  6 \ T_2^1 \times T_2^2
		\end{align*}
		with
		\begin{align*}
		T_2^1
		& = \rho\int_{\brho}^{\rho}\frac{p_\ep(s)}{s^2} - \frac{p_\ep(\brho)}{\brho} (\rho -\brho) \\
		& = (\rho-\brho) \Big[\brho \frac{p_\ep(\brho)}{\brho^2} - \frac{p_\ep(\brho)}{\brho} \Big]
		+ (\rho - \brho)^2 \Big[ \frac{p_\ep(\brho)}{2\brho^2} + \frac{1}{2}\frac{p'_\ep(\brho)}{\brho} - \frac{p_\ep(\brho)}{2 \brho^2}\Big]
		+ \ER \\
		& = \frac{1}{2}\frac{p'_\ep(\brho)}{\brho}(\rho - \brho)^2 + \ER,
		\end{align*}
		\begin{align*}
		T_2^2
		& = p_\ep(\rho) \int_{\brho}^{\rho}\frac{p_\ep(s)}{s^2} = \frac{1}{2}\frac{p'_\ep(\brho)p_\ep(\brho)}{\brho^2} (\rho-\brho)^2 + \ER.
		\end{align*}
		Hence,
		\begin{align*}
		T_2 
		& =  6 \ T_2^1 \times T_2^2 
		= \frac{3}{2}\frac{(p'_\ep(\brho))^2p_\ep(\brho)}{\brho^3} (\rho-\brho)^4 + \ER.
		\end{align*}
		Finally, we get for the last term $T_3$ 
		\begin{align*}
		T_3
		& =  3(u-\bu)^2 
		\left[-\rho \int_{\brho}^{\rho}\left(\frac{p_\ep(s)}{s}\right)^2 
		- (\rho-\brho)\frac{p_\ep(\brho)p_\ep(\rho)}{\brho} 
		+ 2\rho p_\ep(\brho) \int_{\brho}^{\rho}\frac{p_\ep(s)}{s^2}
		\right] \\
		& = 3(u-\bu)^2 (\rho-\brho) \left[
		-\brho \frac{(p_\ep(\brho))^2}{\brho^2} - \frac{(p_\ep(\brho))^2}{\brho}
		+ 2 \brho \frac{(p_\ep(\brho))^2}{\brho^2}
		\right] \\
		& \quad + 3(u-\bu)^2 (\rho-\brho)^2 \Bigg[
		- \frac{(p_\ep(\brho))^2}{\brho^2} - \brho \frac{p'_\ep(\brho) p_\ep(\brho)}{\brho^2}
		+ \frac{(p_\ep(\brho))^2}{\brho^2} - \frac{p'_\ep(\brho) p_\ep(\brho)}{\brho}\\
		& \hspace{4cm} + 2 \frac{(p_\ep(\brho))^2}{\brho^2} + \frac{p'_\ep(\brho) p_\ep(\brho)}{\brho}
		- 2 \frac{(p_\ep(\brho))^2}{\brho^2}  
		\Bigg]
		+ \ER \\
		& =  - 3 \frac{p'_\ep(\brho) p_\ep(\brho)}{\brho} (u-\bu)^2 (\rho-\brho)^2
		+ \ER .
		\end{align*}
		Altogether, we obtain 
		\begin{align}\label{eq:nu-prod}
		< \nu, \teta_3 \tq_4 - \teta_4 \tq_3>
		& = \frac{3}{2} \brho p_\ep(\brho) < \nu, (u-\bu)^4 >
		+ \frac{3}{2}\frac{(p'_\ep(\brho))^2p_\ep(\brho)}{\brho^2} < \nu, (\rho-\brho)^4> \\
		& \quad  - 3 \frac{p'_\ep(\brho) p_\ep(\brho)}{\brho} < \nu,(u-\bu)^2 (\rho-\brho)^2>
		+ \ \ER. \nonumber
		\end{align}
		
		\medskip
		On the other hand, we have from Lemma \ref{lem:exp-eta-q}
		\begin{align*}
		< \nu,\teta_3> \ <\nu,\tq_4> 
		& = \dfrac{3}{2} \brho \big(p_\ep(\brho) + \brho \bu^2\big) \big(<\nu,(u-\bu)^2>\big)^2 \\
		& \quad +  \dfrac{3}{2} \dfrac{(p'_\ep(\brho))^2}{\brho^3}\big(p_\ep(\brho) + \brho \bu^2\big) \big(<\nu,(\rho-\brho)^2>\big)^2 \\
		& \quad + 3 \dfrac{p'_\ep(\brho)}{\brho}\big(p_\ep(\brho) + \brho \bu^2\big) <\nu,(u-\bu)^2> <\nu,(\rho-\brho)^2> \\
		& \quad + 3\brho \bu \left(\dfrac{p_\ep(\brho)}{\brho} + p'_\ep(\brho)\right) <\nu,(u-\bu)^2> <\nu,(\rho-\brho)(u-\bu)> \\
		& \quad + 3 \bu \dfrac{p'_\ep(\brho)}{\brho} \left(\dfrac{p_\ep(\brho)}{\brho} + p'_\ep(\brho)\right) <\nu,(\rho-\brho)^2> <\nu,(\rho-\brho)(u-\bu)> \\
		& \quad + \ER,
		\end{align*}
		while
		\begin{align*}
		< \nu,\teta_4> \ <\nu,\tq_3> 
		& = \dfrac{3}{2} \brho^2 \bu^2 \big(<\nu,(u-\bu)^2>\big)^2 \\
		& \quad +  \dfrac{3}{2} \dfrac{(p'_\ep(\brho))^2}{\brho^2}\bu^2 \big(<\nu,(\rho-\brho)^2>\big)^2 \\
		& \quad + 3  p'_\ep(\brho)\bu^2 <\nu,(u-\bu)^2> <\nu,(\rho-\brho)^2> \\
		& \quad + 3\brho \bu \left(\dfrac{p_\ep(\brho)}{\brho} + p'_\ep(\brho)\right) <\nu,(u-\bu)^2> <\nu,(\rho-\brho)(u-\bu)> \\
		& \quad + 3 \bu \dfrac{p'_\ep(\brho)}{\brho} \left(\dfrac{p_\ep(\brho)}{\brho} + p'_\ep(\brho)\right) <\nu,(\rho-\brho)^2> <\nu,(\rho-\brho)(u-\bu)> \\
		& \quad + 6 \dfrac{p'_\ep(\brho) p_\ep(\brho)}{\brho}  ( <\nu,(\rho-\brho)(u-\bu)>)^2 \\
		& \quad + \ER.
		\end{align*}
		Hence
		\begin{align}\label{eq:prod-nu}
		& < \nu,\teta_3> \ <\nu,\tq_4> - < \nu,\teta_4> \ <\nu,\tq_3> \nonumber \\
		& = \dfrac{3}{2} \brho p_\ep(\brho) \big(<\nu,(u-\bu)^2>\big)^2 
		+ \dfrac{3}{2} \dfrac{(p'_\ep(\brho))^2p_\ep(\brho) }{\brho^3}\big(<\nu,(\rho-\brho)^2>\big)^2  \nonumber\\
		& \quad + 3 \dfrac{p'_\ep(\brho)p_\ep(\brho)}{\brho} <\nu,(u-\bu)^2> <\nu,(\rho-\brho)^2> \nonumber \\
		& \quad - 6 \dfrac{p'_\ep(\brho) p_\ep(\brho)}{\brho}  ( <\nu,(\rho-\brho)(u-\bu)>)^2
		+ \ER.
		\end{align}
		Now, as a consequence of the Div-Curl lemma, we have
		\[
		< \nu, \teta_3 \tq_4 - \teta_4 \tq_3> 
		- \ \big(< \nu,\teta_3> \ <\nu,\tq_4> - < \nu,\teta_4> \ <\nu,\tq_3>)= 0 
		\]
		which leads, thanks to \eqref{eq:nu-prod} and \eqref{eq:prod-nu}, to 
		\begin{align*}
		& \frac{3}{2} \brho p_\ep(\brho) < \nu, (u-\bu)^4 >
		+ \frac{3}{2}\frac{(p'_\ep(\brho))^2p_\ep(\brho)}{\brho^{3}} < \nu, (\rho-\brho)^4> \\
		& \qquad  - 3 \frac{p'_\ep(\brho) p_\ep(\brho)}{\brho} < \nu,(u-\bu)^2 (\rho-\brho)^2> \\
		& =\dfrac{3}{2} \brho p_\ep(\brho) \big(<\nu,(u-\bu)^2>\big)^2 
		+ \dfrac{3}{2} \dfrac{(p'_\ep(\brho))^2p_\ep(\brho) }{\brho^3}\big(<\nu,(\rho-\brho)^2>\big)^2  \nonumber\\
		& \quad + 3 \dfrac{p'_\ep(\brho)p_\ep(\brho)}{\brho} <\nu,(u-\bu)^2> <\nu,(\rho-\brho)^2> \nonumber \\
		& \quad - 6 \dfrac{p'_\ep(\brho) p_\ep(\brho)}{\brho}  ( <\nu,(\rho-\brho)(u-\bu)>)^2
		+ \ER.
		\end{align*}
		This corresponds exactly to Equation \eqref{eq:lem2}.
	\end{proof}

		\subsubsection*{Step 2 - Final expression with positive coefficients}
		The goal now is to use the results of Lemma \ref{lem:A2} and Lemma \ref{lem:app-2} in order to get the desired bound in Lemma \ref{lem:calcul-app}, which comes directly from the following result.
		\begin{lemma}\label{lem-final}
			The relation below holds true:
			\begin{align} {\label{eq:lem3}}
			&C_1(\brho) < \nu, (u -\bu)^4 >
			+ C_2(\brho)  < \nu, (\rho -\brho)^4 > 
			\nonumber \\
			& + C_3 (\brho) \big(< \nu, (u-\bu) (\rho -\brho)> \big)^2
			+ \ \ER
			= 0
			\end{align}
			where
			\begin{align*}
			C_1(\brho) & = B_1(\brho) - \frac{B_3(\brho) \brho^2}{A_3(\brho)} 
			= \ep \dfrac{3-\gamma}{2(\gamma+1)} \dfrac{\brho^{\gamma+1}}{(1-\brho)^\gamma}, \\ 
			C_2(\brho) & = \frac{B_3(\brho) A_1(\brho)}{A_3(\brho)} + B_2(\brho) 
			= \ep^3 \dfrac{\gamma^2(5\gamma +1)}{2(\gamma+1)}  \dfrac{\brho^{3\gamma-5}}{(1-\brho)^{3\gamma+2}}, \\ 
			C_3(\brho) & =2B_3(\brho)
			= \ep^2 \ 6 \gamma  \dfrac{\brho^{2(\gamma-1)}}{(1-\brho)^{2\gamma +1}}.
			\end{align*}
		\end{lemma}
		
		\begin{proof}
			Let us now doing the combination
			\[
			\frac{B_3(\brho)}{A_3(\brho)} \times \eqref{eq:lem1} + \eqref{eq:lem2}
			\]
			which allows to get rid of the cross product terms
			\[
			<\nu,(u-\bu)^2> <\nu,(\rho-\brho)^2> \quad \text{and} \quad < \nu,(u-\bu)^2 (\rho-\brho)^2>.
			\]
			The resulting equation is then
			\begin{equation*}
			\begin{aligned}
			& \frac{B_3(\brho)}{A_3(\brho)} \Bigg(A_1(\brho) <\nu, (\rho-\brho)^4> + A_2(\brho) (<\nu, (\rho-\brho)^2>)^2\\
			& \hspace{1.5cm} + A_4(\brho) (<\nu, (u-\bu)^2>)^2 - \brho^2 <\nu, (u-\bu)^4>\Bigg)\\
			& + B_1(\brho) \Big( <\nu, (u-\bu)^4> - (<\nu, (u-\bu)^2>)^2 \Big) \\
			& + B_2(\brho) \Big( < \nu, (\rho-\brho)^4> - (<\nu, (\rho-\brho)^2>)^2 \Big) \\
			&  + 2B_3(\brho) (<\nu, (\rho-\brho) (u-\bu)>)^2 + \ \ER
			= 0 .
			\end{aligned}
			\end{equation*}			
			which rewrites as
			\begin{align}\label{eq:AiBi}
			& \left[\frac{B_3(\brho) A_1(\brho)}{A_3(\brho)} + B_2(\brho) \right]<\nu, (\rho-\brho)^4>
			+ \left[B_1(\brho) - \frac{B_3(\brho) \brho^2}{A_3(\brho)}\right]<\nu, (u-\bu)^4> \nonumber \\
			& + \left[\frac{B_3(\brho) A_2(\brho)}{A_3(\brho)} + B_2(\brho) \right] (<\nu, (\rho-\brho)^2>)^2
			+ \left[\frac{B_3(\brho) A_4(\brho)}{A_3(\brho)} - B_1(\brho) \right] (<\nu, (u-\bu)^2>)^2  \nonumber \\
			&  + 2B_3(\brho) (<\nu, (\rho-\brho) (u-\bu)>)^2 + \ \ER
			= 0 .
			\end{align}

			From now on, we replace the coefficients $A_i, B_i,$ and use the explicit definition \eqref{df:pressure} of the pressure $p_\ep$:
			\begin{align} \label{eq:HO}
			&  \ep^3 \dfrac{\gamma^2(5\gamma +1)}{2(\gamma+1)}  \dfrac{\brho^{3\gamma-5}}{(1-\brho)^{3\gamma+2}} < \nu, (\rho -\brho)^4> 
			+ \ep \dfrac{3-\gamma}{2(\gamma+1)} \dfrac{\brho^{\gamma+1}}{(1-\brho)^\gamma} < \nu, (u -\bu)^4 > \nonumber \\
			& + \ep \dfrac{3(3-\gamma - 4 \brho)}{2(\gamma+1)} \dfrac{\brho^{\gamma +1}}{(1-\brho)^\gamma} \big(< \nu, (u-\bu)^2>\big)^2 \nonumber  \\
			&  
			+ \ep^3 \dfrac{3\gamma^2(\gamma-3 + 4\brho)}{2(\gamma+1)}  \dfrac{\brho^{3\gamma-5}}{(1-\brho)^{3\gamma+2}} \big(< \nu, (\rho-\brho)^2>\big)^2 \nonumber \\
			& + \ep^2 \ 6 \gamma  \dfrac{\brho^{2(\gamma-1)}}{(1-\brho)^{2\gamma +1}} \big(< \nu, (u-\bu) (\rho -\brho)> \big)^2  + \ \ER
			= 0.
			\end{align}
			For $\gamma \in (1,3)$, we observe that the terms that are problematic are the terms involving 
			\[
			(< \nu, (u-\bu)^2>)^2\quad \text{and} \quad (< \nu, (\rho-\brho)^2>)^2
			\]
			whose multiplicative coefficients may be negative.
			This is an important novelty compared to the isentropic case $p(\rho)= \kappa \rho^\gamma$ treated by Lu in \cite{lu2002} (Section 8.4, Eq (8.4.32)).
			In that latter case, the only negative term was the one involving $(< \nu, (\rho-\brho)^2>)^2$, but the coefficients were such that it was possible to absorb this negative contribution using the inequality
			\[
			(< \nu, (\rho-\brho)^2>)^2 \leq \ < \nu, (\rho-\brho)^4>.
			\]
			In our case, the coefficient in front of $ < \nu, (u-\bu)^4>$ is not large enough to absorb the contribution coming from $(< \nu, (u-\bu)^2>)^2$. 
			Indeed, we have
			\[
			\ep \dfrac{3-\gamma}{2(\gamma+1)} \dfrac{\brho^{\gamma+1}}{(1-\brho)^\gamma}
			+ \ep \dfrac{3(3-\gamma - 4 \brho)}{2(\gamma+1)} \dfrac{\brho^{\gamma +1}}{(1-\brho)^\gamma}
			= \ep \dfrac{2}{\gamma+1}(3-\gamma -3\brho) \dfrac{\brho^{\gamma+1}}{(1-\brho)^\gamma}
			\]
			which may be negative for large values of $\brho \in [0,1]$.
			We overcome this difficulty by using again an identity provided by Div-Curl Lemma which shows that the two problematic terms in $(< \nu, (u-\bu)^2>)^2$ and $(< \nu, (\rho-\brho)^2>)^2$ actually compensate each other.
			The Div-Curl Lemma applied to the pairs $(\teta_1,\tq_1)$ and $(\teta_2,\tq_2)$ yields indeed
			\[
			<\nu , \teta_1 \tq_2 - \teta_2 \tq_1>
			= <\nu,\teta_1> <\nu,\tq_2> - <\nu,\teta_2><\nu,\tq_1>
			\]
			which rewrites, since $<\nu,\teta_1> = <\nu, \tq_1> =0$, as
			\[
			0 
			= <\nu , \teta_1 \tq_2 - \teta_2 \tq_1>
			= <\nu,(p_\ep(\rho)-p_\ep(\brho))(\rho-\brho) - (u-\bu)\brho \rho >.
			\]
			Performing a Taylor expansion, we deduce that
			\begin{align*}
			<\nu,(u-\bu)^2> 
			& = \dfrac{p'_\ep(\brho)}{\brho^2} <\nu,(\rho-\brho)^2> + \ <\nu,O_3> \\
			& = \ep \gamma \dfrac{\brho^{\gamma-3}}{(1-\brho)^{\gamma+1}} <\nu,(\rho-\brho)^2> +  \ <\nu,O_3>
			\end{align*}
			and therefore
			\begin{equation}\label{eq:u2}
			(<\nu,(u-\bu)^2>)^2 
			= \ep^2 \gamma^2 \dfrac{\brho^{2(\gamma-3)}}{(1-\brho)^{2(\gamma+1)}} (<\nu,(\rho-\brho)^2>)^2 + \ER .
			\end{equation}
			Hence
			\begin{align*}
			& \ep \dfrac{3(3-\gamma - 4 \brho)}{2(\gamma+1)} \dfrac{\brho^{\gamma +1}}{(1-\brho)^\gamma} (<\nu,(u-\bu)^2>)^2 \\
			& = \ep^3 \dfrac{3\gamma^2}{2(\gamma+1)} (3-\gamma-4\brho) \dfrac{\brho^{3\gamma-5}}{(1-\brho)^{3\gamma+2}} (<\nu,(\rho-\brho)^2>)^2 + \ER .
			\end{align*}
			and equation \eqref{eq:HO} finally simplifies as 
			\begin{align}
			& \ep \dfrac{3-\gamma}{2(\gamma+1)} \dfrac{\brho^{\gamma+1}}{(1-\brho)^\gamma} < \nu, (u -\bu)^4 >
			\nonumber  \\
			&  + \ep^3 \dfrac{\gamma^2(5\gamma +1)}{2(\gamma+1)}  \dfrac{\brho^{3\gamma-5}}{(1-\brho)^{3\gamma+2}} < \nu, (\rho -\brho)^4 >  \nonumber \\
			& + \ep^2 \ 6 \gamma  \dfrac{\brho^{2(\gamma-1)}}{(1-\brho)^{2\gamma +1}} \big(< \nu, (u-\bu) (\rho -\brho)> \big)^2  + \ \ER
			= 0
			\end{align}
			which achieves the proof.
		\end{proof}

\subsection{Incompressibility and congestion constraint}	
This last appendix section is devoted to the proof of the following lemma.
\begin{lemma}
Let $T>0$ and $(v,u) \in W^{1, \infty}((0,T) \times \R) \times L^\infty((0,T); W^{1,\infty}(\R))$ satisfying
\begin{equation}\label{eq:mass-lim}
\partial_t v = \partial_x u , \quad v_{t=0} = v^0 \quad a.e.
\end{equation}
The following two assertions are equivalent:
\begin{enumerate}
    \item $v(t,x) \geq 1$ for all $(t,x) \in [0,T] \times \R$;
    \item $\partial_x u = 0$ a.e. on $\{v \leq 1\}$ and $v^0 \geq 1$.
\end{enumerate}
\end{lemma}

\medskip
\begin{proof}
\begin{itemize}
    \item $1 \implies 2$:
    Let us assume that $v(t,x) \geq 1$ for all $(t,x) \in [0,T] \times \R$ and introduce $b(v) = v^{-k}$, $k \in \mathbb{N}^*$.
    Multiplying Equation~\eqref{eq:mass-lim} by $b'(v) = -k v^{-(k+1)}$, we get the equation
    \[
    \partial_t v^{-k} = -k v^{-(k+1)}\partial_x u \quad \text{a.e.}.
    \]
    Now, since $v\geq 1$, the sequence $(v^{-k})_{k \in \mathbb{N}^*}$ is bounded in $L^\infty((0,T) \times \R)$, so that $(\partial_t v^{-k})_{k \in \mathbb{N}^*}$ is bounded in $W^{-1,\infty}((0,T); L^\infty(\R))$.
    As a consequence, $\big(-kv^{-(k+1)}\partial_x u\big)_{k \in \mathbb{N}^*}$ is bounded in $W^{-1,\infty}((0,T); L^\infty(\R))$ and, as $k\to +\infty$, we get that
    \[
    v^{-(k+1)}\partial_x u \rightharpoonup 0 \quad \text{in} ~~ \mathcal{D}'.
    \]
    Since, on the other hand,
    \[
    v^{-(k+1)}\partial_x u \rightarrow \mathbf{1}_{\{v=1\}} \partial_x u \quad \text{a.e.},
    \]    
    we deduce that
    \[
    \mathbf{1}_{\{v=1\}} \partial_x u = 0 \quad \text{a.e.}
    \]

    \item $2 \implies 1$: 
    Let $d = v-1$ which satisfies the equation
    \[
    \partial_t d = \partial_x u.
    \]
    Multiply now this equation by $b'_\delta (d)$ where
    \[
    b_\delta(d) = \begin{cases}
    \ [d]_- \quad & \text{if} \quad |d| > \delta \\
    \ \dfrac{1}{4\delta} (d-\delta)^2 \quad & \text{if} \quad |d|\leq \delta
    \end{cases}
    \]
    is a regularization (around $0$) of the function $b :  d \mapsto [d]_- = \max(0, -d)$, 
    \[
    \partial_t b_\delta(d) = b'_\delta(d) \partial_x u \quad \text{a.e.}.
    \]
    As $\delta \to 0$, we can pass to the limit in the above equation and get
    \[
    \partial_t b(d) = B(d) \partial_x u \quad \text{a.e. with} \quad  B(d) =
    \begin{cases}
    -1 \quad & \text{if} \quad d < 0 \\
    -\frac{1}{2} \quad & \text{if} \quad d = 0 \\
    0 \quad & \text{if} \quad d > 0
    \end{cases}
    \]
    Using our assumption on $\partial_x u$, we infer that
    \[
    \partial_t b(d) = 0 \quad \text{a.e.}
    \]
    with initially $b(d)_{|t=0} = 0$ thanks to our assumption on $v^0$.
    Hence, 
    \[
    b(d) = 0 \quad \text{a.e.}~ (t,x), ~\text{i.e.} \quad d = v-1 \geq 0  \quad \text{a.e.}~ (t,x).
    \]
\end{itemize}
\end{proof}
		
\section*{Acknowledgement}
This project has received funding from the European Research Council (ERC) under the European Union's Horizon 2020 research and innovation program Grant agreement No 637653, project BLOC ``Mathematical Study of Boundary Layers in Oceanic Motion’’. This work was  supported by the SingFlows project, grant ANR-18-CE40-0027 of the French National Research Agency (ANR). RB was partially supported by the GNAMPA group of INdAM (GNAMPA project 2019).
The authors thank Roberto Natalini for an introduction on compensated compactness.

\bibliography{biblio_final}	
\end{document}